\documentclass{amsart}
\usepackage{va}
\usepackage{tikz-cd}
\usepackage{bm}
\usepackage{needspace}

\DeclareMathAlphabet{\mathdutch}{U}{dutchcal}{m}{n}

\newcommand\ff{{\mathdutch{f}}}
\newcommand\fancyK{\mathdutch{K}}

\definecolor{LinkBlue}{HTML}{1F4E79}
\definecolor{CiteWine}{HTML}{7A1E3A}
\definecolor{UrlTeal}{HTML}{006D77}

\usepackage[ 
  colorlinks=true,
  linkcolor=LinkBlue,
  citecolor=CiteWine,
  urlcolor=UrlTeal,
  bookmarksnumbered=true,
  bookmarksopen=true,
  unicode=true
]{hyperref} 

\usepackage{todonotes}

\newcommand\CM{\operatorname{CM}}

\newcommand\var{\operatorname{var}}
\newcommand\dR{\mathbf R}

\newcommand\cSP{\mathcal S\mathcal P}

\newcommand\ba{{\bm{a}}}
\newcommand\bb{{\bm{b}}}
\newcommand\be{\bm{e}}
\newcommand\beps{{\bm{\epsilon}}}
\newcommand\bh{\bm{h}}
\newcommand\boldm{\bm{m}}

\newcommand\bzero{{\bm{0}}}

\newcommand\stpull{\bar f^{\#}}
\newcommand\stpush{\bar f_{\#}}

\newcommand\Coh{\operatorname{Coh}}
\newcommand\Vect{\operatorname{Vec}}
\newcommand\perfect{{\operatorname{perf}}}
\newcommand\Perf{\operatorname{Perf}}
\newcommand\ch{\operatorname{ch}}
\newcommand\td{\operatorname{td}}
\newcommand\Td{\operatorname{Td}}
\newcommand\id{\operatorname{id}}
\newcommand\ocM{\overline{\cM}}
\newcommand\ocA{\overline{\cA}}

\makeatletter
\def\l@subsection{\@tocline{2}{0pt}{2pc}{5pc}{}}
\makeatother

\title{Tautological classes on KSBA spaces} 
\date{November 4, 2026} 
\author{Valery Alexeev}
\email{valery@uga.edu}
\address{Department of Mathematics, University of Georgia, Athens, GA 30602, USA}
\begin{document}

\begin{abstract}
  We extend the definitions of kappa, psi and lambda classes from the
  case of moduli spaces of stable curves to the higher-dimensional
  case of KSBA spaces, and prove a number of their basic properties. 
\end{abstract}

\maketitle
\tableofcontents

\section{Introduction}
\label{sec:intro}

The modern enumerative geometry of the moduli spaces of stable curves
$\ocM_{g,n}$ originated in the paper
\cite{mumford1983towards-enumerative} of Mumford, in which he
introduced tautological $\kappa$, $\lambda$, and $\psi$ classes on
$\ocM_{g,n}$. The intersection numbers of kappa and psi classes were
the subject of Witten's conjecture \cite{witten1991two-dimensional},
proved by Kontsevich \cite{kontsevich1992intersection-theory}. The
study of the Chow ring of $\ocM_{g,n}$ and of its tautological subring
has continued in \cite{keel1992intersection-theory,
  looijenga1995on-tautological-ring, arbarello1996combinatorial,
  faber1999conjectural-description, faber2000hodge-integrals,
  ekedahl2001hurwitz-numbers, graber2005relative-virtual-localization,
  mirzakhani2007weil-petersson-volumes,
  pandharipande2015relations-3-spin,
  janda2017relations-equivariant-gromov-witten}
and many other works. It is a very rich and well-developed field.

The higher-dimensional generalizations of moduli of stable curves are
KSBA spaces which parameterize stable pairs
$(X, D = \sum_{i=1}^n a_i D_i)$ consisting of a projective variety and
a $\bQ$-divisor $D$ with rational coefficients $0< a_i\le1$, with slc
singularities and an ample log canonical divisor $K_X+D$.
The basic fact (see Section~\ref{sec:ksba-stack}) is that over the
field of complex numbers for fixed $\ba=(a_1,\dotsc, a_n)$, dimension
$d$ and volume $v=(K_X+D)^d$ there is a Deligne-Mumford stack
$\cSP(\ba,d,v)$ with a projective moduli space $SP(\ba,d,v)$
parameterizing KSBA-stable pairs $(X,D)$, which is a direct
higher-dimensional generalization of $\ocM_{g,n}$ to higher dimensions.

It certainly occurred to many people that one should extend the
enumerative geometry of $\ocM_{g,n}$ to this more general setting. A
fresh impetus was provided by the paper
\cite{donaldson2020fredholm-topology} of Donaldson. Here, we
make a step in this direction.

The curve case has a number of very nice features, such as:
\begin{enumerate}
\item $\ocM_{g,n}$ is a smooth Deligne-Mumford stack.
\item The universal family $\pi_{g,n}\colon \cC_{g,n}\to\ocM_{g,n}$ is
  isomorphic to $\ocM_{g,n+1}$, which is a space of the same type. 
\item $\pi_{g,n}$ is an lci (locally complete intersection) morphism.
\item The boundary of $\ocM_{g,n}$ is a normal crossing divisor whose
  every irreducible component is glued from similar moduli stacks
  $\ocM_{g',n'}$.
\item The determinant of the relative cotangent complex is
  \(\omega_{\pi_{g,n}}\), and one has
  \(c_1(\omega_{\pi_{g,n}}(\mathcal D_{g,n}))=\psi_{n+1}\).
  (Here, $\cD_{g,n}$ is the sum of the sections.)
  This is very convenient for Grothendieck-Riemann-Roch computations.
\end{enumerate}
These features do not hold in general for the KSBA spaces. As a rule,
these moduli spaces are quite singular, the universal families are not
KSBA moduli themselves, the fibers are not lci, and the boundaries are
hard to describe.
Nevertheless, we 
define $\kappa$ and $\psi$ classes and a certain
generalization of $\lambda$ classes on the KSBA spaces.

\smallskip

We do not discuss virtual fundamental classes in this paper. We are
particularly interested in the main components of moduli spaces of
Calabi-Yau pairs and hyperplane arrangements. In these cases, on
an irreducible main
component carrying a perfect obstruction theory that is unobstructed
at the general point, the virtual fundamental class agrees with the
ordinary fundamental class, even if the boundary contains obstructed
points. In any case, even without obstruction theory, integrating
tautological classes against the ordinary fundamental class provides
interesting new invariants.

\smallskip

One object that comes naturally with any KSBA family
$f\colon (X,D)\to S$ is a relatively ample line bundle $L_N$
corresponding to the sheaf $\omega_{X/S}^{[N]}(ND)$ for some large
divisible positive integer $N$, and the corresponding $\bQ$-line
bundle $(1/N)[L_N]$ which does not depend on $N$.  Denote this
$\bQ$-line bundle by $\cO_{X}(K_f + D)$ and its first Chern class by
$K_f+D\in A^1(X)_\bQ$. We define the kappa classes by the formula
\begin{displaymath}
  \kappa_r(f) := f_* \bigl( (K_f+D)^{d+r} \bigr) \in A^r(S)_\bQ
\end{displaymath}
as a direct generalization of Mumford's kappa classes on
$\ocM_{g,n}$. It takes a minute to realize that this formula still
makes sense even though the nice properties of the morphism $\cC_{g,n}
\to \ocM_{g,n}$ are lacking. Indeed, 
by \cite{fulton1998intersection-theory} there is  a
well-defined Gysin pushforward in operational Chow cohomology for any
proper, flat morphism, without requiring $f$ to be lci or $S$ to be
smooth.

The first kappa class $\kappa_1$ is a multiple of a well-known log CM
line bundle. Note: in \cite{alexeev2025kappa-classes} we considered
kappa classes $\kappa_r(S)$ in Chow \emph{homology}; using the
operational Chow cohomology classes $\kappa_r(f)$ above, one has
$\kappa_r(S) = \kappa_r(f)\cap [S]$.

We also define boundary kappa classes in Definition~\ref{def:boundary-kappa}.

In Section~\ref{sec:properties} we prove a number of basic properties
of kappa classes, such as functoriality and
the projection formula, behaviour in product families, under crepant maps
and normalizations. 

The results of Patakfalvi-Xu \cite{patakfalvi2017ampleness-cm} imply
that the classes $\kappa_r$ are nef (numerically nonnegative) and that
$\kappa_1$ is ample on the moduli stack. Thus, for any family
$f\colon X\to S$ over a proper variety $S$, $\kappa_1(f)^r=0$ if and
only if $r>\var f$, the variation of $f$. On the other hand, we show
that $\kappa_r(f)=0$ if and only if $r>\var^\nu f$, the normalized
variation of $f$.
In Theorem~\ref{thm:monomial-positivity} and
Corollary~\ref{cor:monomial-numerical-positivity}
we give a criterion for a monomial in kappa
classes to be numerically trivial, equivalently trivial.

We also investigate the behaviour of kappa classes when the
coefficient $\ba = (a_1, \dotsc, a_n)$ in the pairs
$(X, \sum a_i D_i)$ varies. The papers
\cite{ascher2023wall-crossing,meng2023mmp-for} under some
conditions define reduced moduli stacks $\cM_\ba$ which are certain
simplified versions of the KSBA moduli stacks $\cSP(\ba,d,v)$, and
prove a number of important results about them. The main one is the
existence of a finite polyhedral chamber decomposition on the set of
the coefficients with the following properties: If $\ba,\ba'$ lie in
the same chamber then $\cM_\ba\simeq\cM_{\ba'}$; and if $\ba'$ lies in
a face of the closure of the chamber containing $\ba$ then there is a
natural wall-crossing morphism $\cM_{\ba} \to \cM_{\ba'}$. There are also
stronger statements for the normalizations $\cN_\ba$ of $\cM_\ba$.
We prove that on each chamber the kappa classes on $\cM_\ba$ are
polynomials in the coefficient~$\ba$, and that they are compatible
with wall crossings.


\medskip

Section~\ref{sec:psi-classes} deals with psi classes and mixed classes.
Let $f\colon (X,D=\sum a_iD_i)\to S$ be a family of KSBA pairs. If one
of the divisors $D_i$ is relatively $\bQ$-Cartier, then we can
consider it as an element of $A^1(X)_\bQ$ and define the generalized psi
classes as the Gysin pushforwards
\begin{displaymath}
  \psi_{i,r}(f) = (-1)^r f_* \bigl( D_i^{d+r} \bigr) \in A^r(S)_\bQ.
\end{displaymath}
On $\ocM_{g,n}$ one has $\psi_{i,r} = \psi_i^r$. In higher dimensions,
$\psi_{i,r}$ need not equal $\psi_{i,1}^{\,r}$ for $r\ge2$.

The chamber structure on the set of the moduli spaces $\cM_\ba$
implies that each divisor $D_i$ is automatically relatively
$\bQ$-Cartier if the coefficient $\ba$ can vary in a single chamber
along the $a_i$-direction. Thus, decreasing $a_i$ to $a_i-\epsilon$
allows one to define the psi classes $\psi_{i,r}$ on $\cM_{\ba-\epsilon\be_i}$,
where $\be_i$ is the $i$-th standard basis vector. Under
certain natural technical conditions,
\cite[Theorem~1.9]{ascher2023wall-crossing} and
\cite[Theorem~1.3]{meng2023mmp-for} say that the adjacent wall
morphisms give the identification of the normalizations
$\cN_\ba = \cN_{\ba-\beps}$ for sufficiently small
nonnegative~$\beps$. Using this identification, we define $\psi_{i,r}$
on $\cN_\ba$ itself without requiring $D_i$ to be $\bQ$-Cartier.

More generally, we can define an entire collection of mixed kappa-psi
classes
\begin{displaymath}
  \mu(f;m,\boldm)
    :=f_*\Bigl((K_f+D)^m\prod_{i=1}^nD_i^{m_i}\Bigr)
    \in A^r(S)_\bQ.
\end{displaymath}
Here $r=m+\sum_i m_i-d\ge0$, and we assume that $D_i$ is relatively
$\bQ$-Cartier whenever $m_i>0$.  On $\ocM_{g,n}$ the only nonzero
classes among these are $\kappa_r$ and $\pm \psi_i^r$ since 
$(K_f+D)D_i=0$ and $D_iD_j=0$ for $i\ne j$.  Similarly to the above,
we also define the mixed classes on $\cN_\ba$ for an arbitrary $\ba$
without requiring $D_i$ to be $\bQ$-Cartier, but this time the
definition depends on a finite choice of an $n$-dimensional chamber
containing~$\ba-\beps$, provided such a chamber exists. We also give
an interpretation of the mixed classes as partial derivatives of
kappa classes $\kappa(f;\ba)$ as $\ba$ varies.

\medskip

Another thing one can do with the line bundle $L_N$ is take derived
pushforwards $$\bV_u = \dR f_* L_N^u\in K_0(\Perf S)$$ which for
$u\gg0$ are represented by vector bundles. In Section~\ref{sec:GRR} we
prove that, for $r\ge1$, $\kappa_r$ is the $r$-th Chern class of a rational
linear combination of $\bV_u$'s, and that up to a
multiple $\kappa_r$ is the
coefficient of $u^{r+d}$ in $\ch_r(\bV_u)$. We also prove
similar statements for the mixed classes.
The Chern classes of
$\bV_u$ can be considered to be the generalized, ``pluri log'' lambda
classes, and the identities between $\kappa_r$ and $\bV_u$ are similar
to some of the tautological identities between $\kappa$ and $\lambda$
classes on $\ocM_{g,n}$.  

For the universal pointed stable curve $\pi=\pi_{g,n}$, take $N=u=1$
and write $\bE=\pi_*\omega_\pi$ for the Hodge bundle. The residue
sequence along the marked sections and relative duality give
\begin{displaymath}
  [\dR\pi_*\omega_\pi(\cD_{g,n})]=[\bE]+(n-1)[\cO]
  \qquad\text{in }K_0(\Perf\ocM_{g,n}).
\end{displaymath}
Thus $c_j(\bV_1)=c_j(\bE)=\lambda_j$ for every $j\ge1$, including
when $n=0$.

As in \cite{mumford1983towards-enumerative}, the idea of the proof for
such identities is to use a relative Riemann-Roch theorem. However,
since the KSBA moduli stack is usually not smooth and the fibers are
not lci, the ordinary Grothendieck-Riemann-Roch does not
apply. However, the Baum-Fulton-MacPherson singular Riemann-Roch
theorem still works. We take care to prove identities in
operational Chow cohomology and not just in Chow homology groups.

We note that in some cases, e.g. the KSBA compactification $\ocA_g$ of
moduli spaces of abelian varieties as in
\cite{alexeev1999on-mumfords-construction,alexeev2002complete-moduli},
one can also define the ordinary Hodge bundle
$\bE = (R^1f_*\cO_X)^\vee$. In the case of $\ocA_g$ it is the
canonical extension of the Hodge bundle on $\cA_g$, and its Chern
classes are the ``ordinary'' lambda classes. See e.g.
\cite{FaltingsChai1990, geer1998the-chow-ring,
  vanderGeer1999Cycles, ErdenbergerGrushevskyHulek2006,
  ErdenbergerGrushevskyHulek2010, GrushevskyHulek2012,
  MolchoPandharipandeSchmitt2023, CanningMolchoOpreaPandharipande2025}.

\medskip

In Section~\ref{sec:exampe-cy} we compute the kappa and psi classes in
a particular example: an infinite series of moduli stacks of
Sylvester-type Calabi-Yau $n$-folds from
\cite{alexeev2025tower-complete}.  We also note that the kappa classes
on the KSBA moduli spaces of stable toric pairs
$(X, \Delta+\epsilon D)$ were computed by Byeol Han in
\cite{han2026kappa-classes}. Several other cases are going to follow soon.

\medskip

Unless otherwise stated, all schemes and stacks are assumed to be of
finite type over the field $k=\bC$.

\begin{acknowledgements}
  The author was partially supported by NSF grant DMS-2501855.
\end{acknowledgements}

\begin{aiacknowledgements}
  During the development of this paper, the author used GPT-5.6 and
  GPT-6 as an interactive research tool for brainstorming proof
  strategies, criticizing and improving arguments, and identifying
  relevant literature.  The author assumes full responsibility for the
  correctness, originality, attribution, and content of the paper.
\end{aiacknowledgements}

\section{Definition of kappa classes}
\label{sec:definition}

\subsection{The KSBA moduli stack}
\label{sec:ksba-stack}

We refer the reader to \cite{kollar2023families-of-varieties} for the
details about KSBA-stable pairs and their moduli, only recalling the
basic facts here.  A KSBA-stable pair consists of a projective variety
$X$ and a formal linear combination $D = \sum_{i=1}^n a_i D_i$ of
effective Weil divisors $D_i$ with rational coefficients $0<a_i\le 1$
such that the pair $(X,D)$ has semi log canonical singularities and
the log canonical divisor $K_X + D$ is ample. (We will restrict to the
case $a_i\in\bQ$; \cite{kollar2023families-of-varieties} also treats
the case $a_i\in\bR$.) In dimension~$1$, when each $D_i$ has degree
one, a KSBA-stable pair is the same as Hassett's stable weighted curve
\cite{hassett2003moduli-spaces}. Further, if all weights $a_i$ are
equal to~$1$, this is the same as an ordinary $n$-pointed stable curve
of Deligne-Mumford and Knudsen.

The definition of a family $f\colon (X,D)\to S$ of KSBA-stable pairs
is the most delicate part of the theory, especially when $D\ne 0$ and
$S$ is non-reduced. Two ingredients are easy to state: the morphism
$X\to S$ is flat, and every fiber is a KSBA-stable pair. The divisors
$D_i$, however, need not be flat over $S$. Koll\'ar overcomes this
technical difficulty by introducing the notion of K-flatness. If the
base $S$ is reduced then any relative Mumford divisor $D_i$ is K-flat
by \cite[(7.4.2)]{kollar2023families-of-varieties} and there are
easier definitions of a family in \cite[4.3 and
4.7]{kollar2023families-of-varieties}. 

Any KSBA family over a scheme $S$ of finite type over $\bC$ comes with
a canonical relatively ample $\bQ$-line bundle, defined as follows.
Let $i\colon U\to X$ denote the inclusion of the open subset whose
restriction to every fiber is the set of points that are smooth or
have at worst double-crossing singularities and at which the divisors $D_i$ are
Cartier. Then there exists a positive integer $N$ such that
$Na_i\in\bZ$ and the sheaf
\begin{displaymath}
   L_N=\omega_{X/S}^{[N]}(ND) := i_* \bigl( \omega_{U/S}^{\otimes
     N}(ND_U) \bigr)
\end{displaymath}
is a line bundle, ample over $S$. We formally define a $\bQ$-line bundle
\begin{displaymath}
  \cO_X(K_f+D) := \frac1N[L_N]\in\Pic(X)_\bQ.
\end{displaymath}
This definition is independent of the choices of the integer $N$.  Now
fix the coefficient vector $\ba = (a_1, \dotsc, a_n)$, the relative
dimension $d$ and the volume $v= (K_X + D)^d\in\bQ_{>0}$. Then the
family of KSBA-stable pairs of numerical type $(\ba,d,v)$ is bounded
and there exists a universal integer $N = N(\ba,d,v)$ that works for
any KSBA-stable family of this type. By
\cite[Theorem~9.17]{kollar2023families-of-varieties} the formation of
$L_N$ commutes with arbitrary base changes $S'\to S$.

\cite[\S8.6, \S8.7]{kollar2023families-of-varieties} constructs the
KSBA moduli stack $\cS\cP(\ba, d, v)$ as a Deligne-Mumford quotient
stack $[H/G]$, where $H$ is a certain scheme and $G$ is a reductive
group acting on $H$ properly with finite stabilizers. Its coarse moduli
space $SP(\ba, d, v)$ is a projective scheme over $\bC$. For the rest
of the paper, we fix the numerical data and denote the KSBA moduli
stack by $\cM$ and its coarse moduli space by $M$.  One exception
is Section~\ref{sec:wall-crossing} in which the coefficient $\ba$ varies.

We use calligraphic letters for stacks and the corresponding ordinary
letters for their coarse moduli spaces. An overline on a morphism or
class denotes its induced or descended counterpart on the coarse spaces.

The line bundles $L_N$ defined above assemble into a relatively
ample line bundle $\cL_N$ on the universal family $\ff\colon\cX\to\cM$.

\subsection{Chow groups}
\label{sec:chow-groups}

For Chow groups of schemes, the standard reference is Fulton
\cite{fulton1998intersection-theory}. Vistoli
\cite{vistoli1989intersection-theory} defined the Chow groups
\emph{with rational coefficients} of DM stacks. With very few
exceptions, the kappa classes are rational rather than integral. Thus
the results of \cite{vistoli1989intersection-theory} suffice for
us. Alternatively, since $\cM$ is a global quotient stack, the results
of Edidin--Graham \cite{edidin1998equivariant-intersection}, also with
rational coefficients, apply. The two approaches
\cite{vistoli1989intersection-theory,
  edidin1998equivariant-intersection} agree where both apply. The more
sophisticated results of Kresch \cite{kresch1999cycle-groups} about
integral Chow groups of Artin stacks go beyond what is needed for our
purposes.

\subsection{Operational Chow cohomology}
\label{sec:chow-schemes}

For an arbitrary, possibly singular scheme of finite type over a field
$k$, the operational Chow groups $A^*$ were defined in
\cite[Chapter~17]{fulton1998intersection-theory}, following
\cite{fulton1981categorical-framework}. For any morphism
$f\colon X\to Y$, a bivariant class $c\in A^p(X\xrightarrow{f} Y)$ is
a rule that associates to each integer $k$ and each morphism
$g\colon Y'\to Y$ with the fiber product
\begin{equation}\label{eq:square}
  \begin{tikzcd}
    X' = X\times_Y Y'
    \arrow[r, "g'{}"]
    \arrow[d, "f'"']
    &
    X
    \arrow[d, "f"]
    \\
    Y'
    \arrow[r, "g"']
    &
    Y
  \end{tikzcd}
\end{equation}
a homomorphism of Chow groups
\begin{displaymath}
  c_g\colon A_k(Y')\to A_{k-p}(X'),
  \qquad
  \alpha \mapsto c \cap \alpha := c_g(\alpha),
\end{displaymath}
compatible with push-forward, pull-back, and generalized Gysin
homomorphisms in the argument $Y'$. For any morphisms $f\colon X\to Y$
and $g\colon Y\to Z$, there is a product
\begin{displaymath}
  A^p(X\xrightarrow{f} Y) \otimes A^q(Y\xrightarrow{g} Z) \to
  A^{p+q}(X\xrightarrow{gf} Z)
\end{displaymath}
making $A^*(X):= \oplus_p A^p(X\xrightarrow{{\rm id}_X} X)$ into a graded
associative ring with identity.
The subring generated by Chern classes of vector bundles lies in the
center of $A^*(X)$. In any case, in characteristic $0$ the ring
$A^*(X)$ is commutative; see \cite[\S 17.3 and Example
17.4.4]{fulton1998intersection-theory}. 

Vistoli \cite[Section~5]{vistoli1989intersection-theory} extended
these definitions, constructions, and results to the bivariant Chow
groups of DM stacks, again with rational coefficients.

Note that by \cite[Corollary A.4]{bloch1995non-archimedean} for a
proper scheme $X$ over $\bC$ there is a functorial graded ring homomorphism
$A^*(X) \to \oplus_p \Gr^W_{2p} H^{2p}(X^{\rm an}, \bQ)$,
cf. \cite{totaro2014chow-groups}.

\subsection{Standard constructions}
\label{sec:standard}

In the case of schemes, see \cite[Chapter~3 and
\S17.3]{fulton1998intersection-theory} for (a). For the bivariant Chow
formalism, including flat orientations, products, and proper
pushforward of bivariant classes used in (b) and (c), see
\cite[\S\S17.1--17.3]{fulton1998intersection-theory}. The extension to
DM stacks was given in
\cite[Section~5]{vistoli1989intersection-theory}.

\smallskip (a) For any vector bundle $V$ on $X$ and any integer
$m\ge0$, there is a Chern class $c_m(V) \in A^m(X)$.  In particular, a
line bundle $F$ on $X$ defines $c_1(F)\in A^1(X)$.

\smallskip (b) A flat morphism $f\colon X\to Y$ of relative dimension
$d$ defines an \emph{orientation} $[f] \in A^{-d}(X\to Y)$ by flat
pullback: for every $\alpha\in A_k(Y')$, the rule is
\begin{displaymath}
  [f] \cap \alpha = f'{}^*\alpha \in A_{k+d}(X').
\end{displaymath}

\smallskip (c) Any flat and proper morphism $f\colon X\to Y$ defines
a generalized Gysin pushforward homomorphism
\begin{displaymath}
  f_*\colon A^p(X) \to A^{p-d}(Y), \qquad
  f_*(c)=f_*(c\cdot [f]).
\end{displaymath}
Explicitly, for any $Y'\to Y$ as in diagram~\eqref{eq:square}
and $\alpha\in A_k(Y')$, the rule for $f_*c$ is
\begin{displaymath}
  f_*c\cap \alpha =
  f'_*(c\cap f'{}^*\alpha).
\end{displaymath}
Here $\alpha\in A_k(Y')$; $f'{}^*\alpha \in A_{k+d}(X')$ is the flat
pullback defined by the orientation $[f]$ as in (b);
$c\cap f'{}^*\alpha \in A_{k+d-p}(X')$; and
$f'_*(c\cap f'{}^*\alpha) \in A_{k+d-p}(Y')$ is the proper pushforward
in homology.

\begin{remark}
  Following Fulton \cite{fulton1998intersection-theory}, we use the
  notation $f_*c$ for this Gysin pushforward. The same symbol also
  denotes the proper pushforward in homology which may sometimes lead
  to confusion; the intended meaning is determined by whether the
  argument is an operational Chow class or a homology class.
\end{remark}

\smallskip (d) Let $f\colon X\to Y$ be a flat and proper morphism, and
let $L$ be a line bundle on $X$. Constructions (a), (b), and (c)
combine to define an operational Chow cohomology class
\begin{displaymath}
  f_*\bigl( c_1(L)^{r+d} \bigr) \in A^r(Y).
\end{displaymath}

\subsection{Kappa classes and boundary kappa classes}
\label{sec:the-definition}

\begin{definition}\label{def:kappa}
  Let $\ff\colon (\cX,\cD)\to \cM$ be the universal family over a KSBA
  stack and $\cL_N$ be the line bundle as in
  Section~\ref{sec:ksba-stack}. Set
  \begin{displaymath}
    \ell_N=c_1(\cL_N)=N(K_\ff+\cD)\in A^1(\cX)_\bQ.
  \end{displaymath}
  We define
  \begin{displaymath}
    \kappa_r = \ff_* (K_\ff+\cD)^{r+d} =
    \frac1{N^{r+d}}\, \ff_* \ell_N^{r+d} 
    \in A^r(\cM)_\bQ,
  \end{displaymath}
  using construction (d) from the previous section. A scheme $S$ over
  $\cM$ with the morphism $g\colon S\to\cM$ is the same as a family
  $f\colon (X,D)\to S$, and
  \begin{displaymath}
    \kappa_r(f) = \kappa_r|_S = f_* (K_f+D)^{r+d}
    = g^* \kappa_r \in A^r(S)_\bQ.
  \end{displaymath}
  Thus the argument records the family map $f$; for the
  universal family we retain the notation $\kappa_r$. For any base
  change $T\to S$, write $f_T$ for the induced family over $T$.
\end{definition}

\begin{example}[Pointed stable curves] \label{ex:cgn}
  On the moduli stack of $n$-pointed stable curves, with universal
  family 
  \begin{displaymath}
    \ff\colon
    (\cC_{g,n},\cD_1+\dotsb+\cD_n)
      \longrightarrow \overline{\cM}_{g,n},
  \end{displaymath}
  this definition gives
  \begin{displaymath}
    \kappa_r
      =\ff_*
      (K_\ff+\cD_1+\dotsb+\cD_n)^{r+1},
  \end{displaymath}
  the Arbarello-Cornalba \cite{arbarello1996combinatorial} logarithmic
  kappa class. In particular, one has
  \begin{math}
    \kappa_0=2g-2+n.
  \end{math}
  Thus, our convention is the logarithmic one, using the relative
  dualizing sheaf twisted by the marked sections.
\end{example}


The next definition requires a well-defined flat intersection family.
Strong KSBA-stable families, as in
\cite[8.18]{kollar2023families-of-varieties}, provide useful flatness
conditions for divisors with coefficients $a_i>1/2$. The corresponding
moduli stack
$\cM^{\rm strong} = \cS\cP^{\rm strong}(\ba,d,v)$, which is generally
different from $\cM = \cS\cP(\ba,d,v)$, has the same reduced part.
Thus, when one works with the normalization $\cN$ of $\cM$, this
difference disappears. Neither strong stability nor the inequality
$a_i>1/2$ is required for the definition whenever the specified
intersection family is flat.

\begin{definition}\label{def:boundary-kappa}
  For a subset $I\subset\{1,\dotsc, n\}$ of indices, suppose that
  there is a well-defined flat family
  $\ff_I\colon \cX_I = \cap_{i\in I}\cD_i\to\cM$ (or $\cN$) of relative
  dimension $d_I$. With $\cL_{N,I}=\cL_N|_{\cX_I}$, set
  \begin{displaymath}
    \ell_{N,I}=c_1(\cL_{N,I})=\ell_N|_{\cX_I}.
  \end{displaymath}
  Then
  \begin{displaymath}
    (K_\ff+\cD)|_{\cX_I}=\frac1N\ell_{N,I}.
  \end{displaymath}
  We define the boundary kappa classes by
  \begin{displaymath}
    \kappa^{(I)}_r
      =\frac1{N^{d_I+r}}\,(\ff_I)_*\ell_{N,I}^{d_I+r}.
  \end{displaymath}
\end{definition}

Many basic properties of $\kappa_r$ proved in the next section,
such as functoriality, the projection formula, and nonnegativity,
extend immediately to $\kappa_r^{(I)}$.

\begin{example}\label{ex:fedorchuk}
  For the moduli space $\ocM_{g,\ba}$ of weighted stable curves, the
  morphisms $\ff_i\colon \cD_i\to \ocM_{g,\ba}$ for $1\le i\le n$ are
  isomorphisms. In the unweighted case of $\ocM_{g,n}$ one has
  $\omega_\ff(\cD)|_{\cD_i} \simeq \cO_{\cD_i}$, so the boundary kappa
  classes $\kappa^{(i)}_r$ are zero for $r>0$, while
  $\kappa^{(i)}_0=1$. On the other hand, if some $a_i< 1$ then
  $\kappa^{(i)}_1$ is 
  \begin{displaymath}
    c_1\bigl( \omega_\ff(\cD)|_{\cD_i} \bigr) = 
    (a_i-1)\cD_i|_{\cD_i} + \sum_{j\ne i} a_j \cD_j|_{\cD_i} =
    (1-a_i)\psi_i + \sum_{j\ne i} a_j \cD_j\cap\cD_i,
  \end{displaymath}
  which is generally nontrivial. These $\bQ$-line bundles were
  investigated by Fedorchuk \cite{fedorchuk2011moduli-of-weighted}, who
  showed that they are nef. 
\end{example}

\section{Basic properties of kappa classes}
\label{sec:properties}

\subsection{\texorpdfstring{$\kappa_0$ and $\kappa_1$}{kappa 0 and kappa 1}}
\label{sec:kappa01}

\begin{proposition}
  One has
  \begin{displaymath}
    \kappa_0=v\cdot 1\in A^0(\cM)_\bQ,
  \end{displaymath}
  where $v$ is the volume on the chosen component
  $\cM=\cSP(\ba,d,v)$.
\end{proposition}
\begin{proof}
  It is enough to test the pullback of the operational class on every
  integral cycle $[Z]$ on every scheme $T$ equipped with a morphism
  $T\to\cM$. Let $f_T$ be the pulled-back family over $T$.
  After replacing $T$ by $Z$, let $f_Z\colon(X_Z,D_Z)\to Z$
  be the pulled-back family. Since the generic fiber has volume $v$,
  proper pushforward gives
  \begin{displaymath}
    (f_Z)_*\bigl((K_{f_Z}+D_Z)^d\cap[X_Z]\bigr)=v[Z].
  \end{displaymath}
  Hence $\kappa_0(f_T)$ acts on every integral cycle as multiplication
  by $v$, which proves the operational identity.
\end{proof}

Under the normalization used in \cite{wang2014nonexistence-asymptotic},
$\kappa_1$ is the first Chern class
\begin{displaymath}
  \kappa_1=c_1(L_{\CM}^{\rm log})
\end{displaymath}
of the logarithmic CM (Chow--Mumford) rational line bundle. The
descent of $L_{\CM}^{\rm log}$ to the coarse KSBA moduli space
is ample by
\cite[Theorem~1.2 and the proof of Theorem~1.1]
{patakfalvi2017ampleness-cm}.

\subsection{Functoriality and projection formula}

\begin{proposition}[Functoriality]\label{prop:functoriality}
Let $f\colon(X,D)\to S$ be a family, let $g\colon S'\to S$ be a
morphism, and denote by $f'\colon (X',D')\to S'$ the base-changed family.  Then
$g^*\kappa_r(f)=\kappa_r(f')$.
If $h\colon T\to S$ is proper and $\alpha\in A_*(T)_{\bQ}$, then
$h_*\bigl(\kappa_r(f_T)\cap\alpha\bigr) = \kappa_r(f)\cap h_*\alpha$.
\end{proposition}

\begin{proof}
The first equality is the base-change compatibility of the operational Gysin
construction.  The second is the projection formula for operational Chow
classes.
\end{proof}

\begin{corollary}[Generically finite and birational base change]
  \label{cor:base-change-degree}
  Suppose that a morphism
  \begin{math}
    h\colon T\longrightarrow S
  \end{math}
  is proper, dominant, and generically finite of degree~$e$.  Assume
  that $T$ and $S$ are integral schemes of the same
  dimension~$b$. Then, for every $r\geq0$,
  \begin{displaymath}
    h_*\bigl(\kappa_r(f_T)\cap[T]\bigr)
    =e\,\kappa_r(f)\cap[S]
    \qquad\text{in }A_{b-r}(S)_\bQ.
  \end{displaymath}
  If $S$ is proper and $r=b$, then
  \begin{math}
    \deg\kappa_b(f_T)=e\,\deg\kappa_b(f).
  \end{math}
  In particular, top kappa degrees are invariant under proper birational
  modifications of a proper base.
\end{corollary}

\begin{proof}
  Functoriality and the projection formula give
  \begin{displaymath}
    h_*\bigl(\kappa_r(f_T)\cap[T]\bigr)
      =h_*\bigl(h^*\kappa_r(f)\cap[T]\bigr)
      =\kappa_r(f)\cap h_*[T].
  \end{displaymath}
  Since $h_*[T]=e[S]$, this proves the cycle identity. Taking degrees
  when $S$ is proper and $r=b$ gives the numerical statement.
\end{proof}

\subsection{Products}

\begin{definition}[Kappa series]
\label{def:kappa-series}  
For a family $f\colon(X,D)\to S$ of relative dimension~$d$, define the
multiplicative generating function for the kappa classes as
\begin{displaymath}
  \fancyK_f(t)
  := t^{-d} f_* e^{(K_f+D)t} 
  = \sum_{r\geq0}\frac{\kappa_r(f)}{(r+d)!}\,t^r
    \in A^*(S)_\bQ[t].
\end{displaymath}
\end{definition}

\begin{proposition}[Product formula]\label{prop:product-formula}
  Let $f_k\colon(X_k,D_k)\to S$, $k=1,2$, be KSBA-stable families
  of relative dimensions $d_k$, and put
  \begin{displaymath}
    f\colon (X,D):=
    \bigl(X_1\times_S X_2,\,p_1^*D_1+p_2^*D_2\bigr)
    \longrightarrow S.
  \end{displaymath}
  Then
  \begin{displaymath}
    \fancyK_f(t) = \fancyK_{f_1}(t) \cdot \fancyK_{f_2}(t).
  \end{displaymath}
  Term by term, this means that
  \begin{displaymath}
    \kappa_r(f)
    =
    \sum_{j=0}^r
    \binom{d_1+d_2+r}{d_1+j}
    \kappa_j(f_1)\,
    \kappa_{r-j}(f_2).
  \end{displaymath}
  For $r=0$, if $v_k$ is the volume of the $k$-th family, then
  $v=\binom{d_1+d_2}{d_1}v_1v_2$.
\end{proposition}

\begin{proof}
  On the fiber product one has
  \begin{displaymath}
    K_f+D=p_1^*(K_{f_1}+D_1)+p_2^*(K_{f_2}+D_2).
  \end{displaymath}
  Compatibility of exterior products with proper pushforward, checked
  after every base change to $S$, and the projection formula give
  \begin{displaymath}
    f_*\bigl(p_1^*\alpha\cdot p_2^*\beta\bigr)
      =(f_1)_*(\alpha)\,(f_2)_*(\beta).
  \end{displaymath}
  The statement now follows from $d=d_1+d_2$ and the identity
  \begin{displaymath}
    e^{(K_f+D)t} = e^{p_1^*(K_{f_1}+D_1)t} \cdot
    e^{p_2^*(K_{f_2}+D_2)t}.
  \end{displaymath}
\end{proof}

For example, the formulas of this section apply to the moduli spaces
of stable surfaces which are products of stable curves, over the
moduli stack $\ocM_{g_1}\times\ocM_{g_2}$.

\subsection{Crepant maps}

For a proper flat morphism $f\colon X\to S$ of pure relative
dimension~$d$ and a rational line bundle $\Lambda\in\Pic(X)_\bQ$ (not
necessarily $f$-ample),
write
\begin{displaymath}
  \kappa_r(f;\Lambda)
    :=f_*\bigl(c_1(\Lambda)^{r+d}\bigr)\in A^r(S)_\bQ.
\end{displaymath}
For a KSBA family and $\Lambda=\cO_X(K_f+D)$, this is our kappa
class $\kappa_r(f)$.

We first record a criterion ensuring that the degree of a proper
morphism on fundamental cycles is preserved by arbitrary base change.

\begin{lemma}
  \label{lem:universal-cycle-degree}
  Let
  \begin{math}
    f_Y\colon Y\longrightarrow S,
    \ 
    f_X\colon X\longrightarrow S
  \end{math}
  be flat morphisms of pure relative dimension~$d$, let $e\geq1$, and
  let $\rho\colon Y\longrightarrow X$ be a proper morphism
  over~$S$. Suppose that there is an open subscheme $U\subset X$ such
  that, after setting $V:=\rho^{-1}(U)$,
  \begin{enumerate}
  \item the restriction
    \begin{math}
      q:=\rho|_V\colon V\longrightarrow U
    \end{math}
    is finite locally free of constant rank~$e$ or, equivalently, finite
    flat of constant degree~$e$, since $U$ is Noetherian;
  \item $U_s$ is dense in $X_s$ for every point $s\in S$.
  \end{enumerate}
  Then, for every morphism $T\to S$ and every $\alpha\in Z_*(T)$,
  \begin{displaymath}
    (\rho_T)_*f_{Y,T}^*\alpha=e\,f_{X,T}^*\alpha
    \qquad\text{in }Z_*(X_T),
  \end{displaymath}
  where $f_{X,T}$ and $f_{Y,T}$ denote the base changes of $f_X$
  and $f_Y$. In particular,
  \begin{displaymath}
    (\rho_T)_*[Y_T]=e[X_T]
    \qquad\text{in }Z_*(X_T),
  \end{displaymath}
  and hence also in $A_*(X_T)$.
\end{lemma}

\begin{proof}
  By linearity, we may take $\alpha=[Z]$ for an integral closed
  subscheme $Z\subset T$. Compatibility of flat pullback with proper
  pushforward then reduces the first assertion to $T=Z$ and
  $\alpha=[T]$.

  Put $N=\dim T+d$. By flatness and pure relative dimension, $X_T$
  and $Y_T$ are pure $N$-dimensional. Every irreducible component of
  $X_T$ dominates the integral base $T$ by flatness. Fiberwise
  density therefore places its generic point in $U_T$, so the
  restriction $p\colon Z_N(X_T)\to Z_N(U_T)$ is injective. The morphism
  $q_T\colon V_T\to U_T$ is finite locally free of rank~$e$.
  Since proper pushforward commutes with restriction to an open
  subscheme,
  \begin{displaymath}
    \left.(\rho_T)_*[Y_T]\right|_{U_T}
      =(q_T)_*[V_T]
      =e[U_T]
      =\left.e[X_T]\right|_{U_T}.
  \end{displaymath}
  Injectivity of $p$ therefore gives the required equality on $X_T$.
  The second follows by taking $\alpha=[T]$ for arbitrary~$T$, since
  flat pullback preserves fundamental cycles.
\end{proof}

\begin{theorem}[Crepant functoriality for polarized flat families]
  \label{thm:crepant-functoriality}
  Let
  \begin{displaymath}
    f_Y\colon Y\longrightarrow S,
    \qquad
    f_X\colon X\longrightarrow S,
  \end{displaymath}
  be proper flat morphisms of pure relative dimension~$d$, and let
  $e\geq1$.  Let $\rho\colon Y\to X$ be a proper morphism over~$S$,
  and let
  \begin{displaymath}
    \Lambda_X\in\Pic(X)_\bQ
    \quad\text{and}\quad
    \Lambda_Y=\rho^*\Lambda_X   \in\Pic(Y)_\bQ.
  \end{displaymath}
  For a base change $T\to S$, write $f_{X,T}$ and $f_{Y,T}$ for
  the induced morphisms over $T$. Then the following statements hold.
  \begin{enumerate}
    \item If
      \begin{math}
        \rho_*[Y]=e[X]
      \end{math}
      in $A_*(X)_\bQ$,
      then
      \begin{displaymath}
        \kappa_r(f_Y;\Lambda_Y)\cap[S]
          =e\,\kappa_r(f_X;\Lambda_X)\cap[S].
      \end{displaymath}
    \item Suppose that for every morphism $Z\to S$ from an integral
      scheme $Z$ one has
      \begin{math}
        (\rho_Z)_*[Y_Z]=e[X_Z]
      \end{math}
      in $A_*(X_Z)_\bQ$.
      Then
      \begin{displaymath}
        \kappa_r(f_Y;\Lambda_Y)
          =e\,\kappa_r(f_X;\Lambda_X)
        \qquad\text{in }A^r(S)_\bQ.
      \end{displaymath}
  \end{enumerate}
  For example, the hypothesis in (2) holds under the assumptions of
  Lemma~\ref{lem:universal-cycle-degree}.
  For $e=1$, it is
  enough that $\rho$ be an isomorphism over an open subset
  $U\subset X$ whose restriction to every fiber of $X\to S$ is dense.
\end{theorem}

\begin{proof}
  (1)
  Put $\ell:=c_1(\Lambda_X)$ and $m:=r+d$. The pullback hypothesis
  gives $c_1(\Lambda_Y)=\rho^*\ell$, compatibly with every base
  change.
  Under the assumptions of (1), flatness gives
  $f_X^*[S]=[X]$ and $f_Y^*[S]=[Y]$. The projection formula gives
  \begin{displaymath}
    \begin{aligned}
      \kappa_r(f_Y;\Lambda_Y)\cap[S]
        &=(f_Y)_*\bigl(\rho^*\ell^m\cap[Y]\bigr) \\
        &=(f_X)_*\bigl(\ell^m\cap\rho_*[Y]\bigr) \\
        &=e\,(f_X)_*\bigl(\ell^m\cap[X]\bigr) \\
        &=e\,\kappa_r(f_X;\Lambda_X)\cap[S].
    \end{aligned}
  \end{displaymath}

  (2) Apply (1) after every integral base change $Z\to S$.
  For an arbitrary test scheme $T\to S$, the group $A_*(T)_\bQ$
  is generated by the pushforwards of the fundamental classes of
  integral closed subschemes $Z\subset T$. Compatibility with proper
  pushforward therefore gives the asserted operational identity.
\end{proof}

The theorem and its proof also apply over a Deligne--Mumford base: one
tests the asserted operational identity after pullback to schemes.

For KSBA families, the applications we have in mind concern crepant
morphisms $\rho\colon Y\to X$ with the $\bQ$-line bundles
$\Lambda_X=\cO_X(K_{f_X}+D)$ and
$\Lambda_Y=\cO_Y(K_{f_Y}+\Delta)$. One such application appears in the
next section on normalization, and another one is in
Section~\ref{sec:wall-crossing} on wall crossing. For other
applications, see \cite{alexeev2026explicit-compactifications,
  alexeev2025secondary-burniat} for the compactifications of moduli of
Campedelli and Burniat surfaces of general type which are crepant
Galois covers of pairs $(X,\frac12 D)$ for certain del Pezzo surfaces~$X$.

\subsection{Normalization}
\label{sec:normalization}

Let $f\colon(X,D)\to S$ be a family of KSBA-stable pairs over an
integral normal scheme, and let
\begin{displaymath}
  \nu\colon(X^\nu,D^\nu+\Gamma)
     =\bigsqcup_a(X_a,D_a+\Gamma_a)\longrightarrow(X,D)
\end{displaymath}
be the normalization, with conductor $\Gamma$.
Here, $D_a = D^\nu|_{X_a}$, $\Gamma_a = \Gamma|_{X_a}$, and $f_a =
f\circ \nu |_{X_a}$. One has the adjunction formula
\begin{equation}
  \label{eq:normalization-adjunction}
    K_{f\circ\nu}+D^\nu+\Gamma
    =\nu^*(K_f+D).
\end{equation}
By
\cite[Lemma~3.2]{patakfalvi2017ampleness-cm}, each morphism
$f_a\colon(X_a,D_a+\Gamma_a)\to S$ is a family of KSBA-stable pairs.
In particular, each morphism $f_a$ is flat. 

\begin{proposition}[Additivity under normalization]
  \label{prop:normalization-additivity}
  With the notation above, one has
  \begin{displaymath}
    \kappa_r(f)
      =
    \sum_a\kappa_r(f_a)
    \qquad\text{in }A^r(S)_\bQ.
  \end{displaymath}
\end{proposition}
\begin{proof}
  Let $U\subset X$ be the relative smooth locus of $f$, and put
  $V:=\nu^{-1}(U)$. Since the fibers of $f$ are reduced over the
  characteristic-zero field~$\bC$, the open set $U_s$ contains the
  generic point of every irreducible component of $X_s$, and hence is
  dense in $X_s$ for every $s\in S$. The morphism $U\to S$ is
  smooth and $S$ is normal, so $U$ is normal. Consequently,
  \begin{math}
    \nu|_V\colon V\longrightarrow U
  \end{math}
  is an isomorphism.

  The disjoint union $X^\nu=\bigsqcup_aX_a$ and $X$ are proper and
  flat over $S$ of pure relative dimension~$d$. Therefore
  Lemma~\ref{lem:universal-cycle-degree}, with $e=1$, gives
  \begin{displaymath}
    (\nu_T)_*[X^\nu_T]=[X_T]
  \end{displaymath}
  after every base change $T\to S$. Together with the adjunction
  identity~\eqref{eq:normalization-adjunction},
  Theorem~\ref{thm:crepant-functoriality}(2) gives
  \begin{displaymath}
    \kappa_r(f\circ\nu;
      \cO_{X^\nu}(K_{f\circ\nu}+D^\nu+\Gamma))
      =\kappa_r(f).
  \end{displaymath}
  The Gysin pushforward for a finite disjoint union is the sum of the
  pushforwards from its components, so the left side is
  $\sum_a\kappa_r(f_a)$, as required.
\end{proof}

\subsection{Descent to the coarse moduli space}
\label{sec:stack-to-coarse}

Let $\ff\colon (\cX,\cD)\to \cM$ be the universal family, and let
$X$ and $M$ be the coarse moduli spaces of $\cX$
and~$\cM$. By the universal property of coarse moduli spaces,
there is a canonical commutative diagram
\begin{equation}\label{eq:coarse-moduli}
\begin{tikzcd}
  \cX
  \arrow[r, "\pi'"]
  \arrow[d, "\ff"']
  &
  X
  \arrow[d, "\bar f"]
  \\
  \cM
  \arrow[r, "\pi"']
  &
  M.
\end{tikzcd}
\end{equation}
Usually, however, this diagram is not a fiber square, and the morphism
$\bar f$ need not be flat. Applying \cite[Prop.~6.1(i),(ii)]{vistoli1989intersection-theory}
separately to the coarse maps $\pi$ and $\pi'$, one obtains the
graded isomorphisms
\begin{align*}
  \pi_*&\colon A_*(\cM)_\bQ \isoto A_*(M)_\bQ,
  &
  \pi'_*&\colon A_*(\cX)_\bQ \isoto A_*(X)_\bQ,\\
  \pi^*&\colon A^*(M)_\bQ \isoto A^*(\cM)_\bQ,
  &
  \pi'{}^*&\colon A^*(X)_\bQ \isoto A^*(\cX)_\bQ.
\end{align*}

\begin{proposition}[Descent to the coarse moduli space]
  \label{prop:coarse-descent}
  There is a unique class $\bar\kappa_r\in A^r(M)_{\bQ}$ such
  that $\pi^*\bar\kappa_r=\kappa_r$. Every family $f\colon (X,D)\to S$
  of KSBA-stable pairs with coarse moduli map $\bar\mu\colon S\to M$
  satisfies $\kappa_r(f)=\bar\mu^*\bar\kappa_r$.
\end{proposition}
\begin{proof}
  The isomorphism $\pi^*\colon A^r(M)_\bQ\isoto A^r(\cM)_\bQ$
  gives a unique class $\bar\kappa_r$ with
  $\pi^*\bar\kappa_r=\kappa_r$. A family determines a classifying map
  $\mu\colon S\to\cM$, and its coarse moduli map is
  $\bar\mu=\pi\circ\mu$. Therefore
  \begin{displaymath}
    \kappa_r(f)=\mu^*\kappa_r
      =\mu^*\pi^*\bar\kappa_r
      =\bar\mu^*\bar\kappa_r.
  \end{displaymath}
\end{proof}

\subsection{Vanishing}
\label{sec:vanishing}

\begin{definition}\label{def:variation}
  Let $f\colon(X,D)\to S$ be a family of KSBA-stable pairs. The
  \emph{variation} $\var f$ is the dimension of the closure of the
  image of its coarse moduli map $S\to M$.
\end{definition}

\begin{theorem}[Vanishing of kappa polynomials]\label{thm:vanishing}
  Let $r_1,\dotsc,r_m\geq0$. If
  \begin{displaymath}
    r_1+\dotsb+r_m>\var f,
  \end{displaymath}
  then
  \begin{displaymath}
    \kappa_{r_1}(f)\cdots\kappa_{r_m}(f)=0
    \qquad\text{in }A^{r_1+\dotsb+r_m}(S)_\bQ.
  \end{displaymath}
  Consequently, every homogeneous polynomial in the kappa classes of
  total codimension greater than $\var f$ vanishes. In particular,
  \begin{displaymath}
    \kappa_r(f)=0\quad(r>\var f),
    \qquad
    \kappa_1(f)^{\,\var f+1}=0.
  \end{displaymath}
  For the universal family, all kappa polynomials of degree greater
  than $\dim M$ vanish.
\end{theorem}

\begin{proof}
  Let $\bar\mu\colon S\to M$ be the coarse moduli map, let
  $Z\subset M$ be its scheme-theoretic image, and write
  $j\colon Z\hookrightarrow M$. Then $\dim Z=\var f$, and
  Proposition~\ref{prop:coarse-descent} gives
  $\kappa_r(f)=\bar\mu_Z^*j^*\bar\kappa_r$, where $\bar\mu_Z\colon S\to Z$ is
  the induced morphism. Hence
  \begin{displaymath}
    \kappa_{r_1}(f)\cdots\kappa_{r_m}(f)
    =\bar\mu_Z^*\bigl(
    j^*\bar\kappa_{r_1}\cdots j^*\bar\kappa_{r_m}
    \bigr).
  \end{displaymath}
  By \cite[Example~17.3.3]{fulton1998intersection-theory}
  one has $A^p(Z)_\bQ=0$ for $p>\dim Z$, and the class in parentheses
  vanishes when $r_1+\dotsb+r_m>\dim Z$.
\end{proof}

\subsection{Nonnegativity, strict positivity, and numerical triviality}
\label{sec:nonnegativity}

For a proper pure $r$-dimensional scheme $S$, we write
$\deg_S c:=\deg(c\cap[S])$ for $c\in A^r(S)_\bQ$.

\begin{proposition}[Numerical nonnegativity of kappa classes]
  \label{prop:kappa-nonnegative}
  Let $f\colon(X,D)\to S$ be a KSBA-stable family. Then every kappa
  class $\kappa_r(f)$ is numerically nonnegative: for every proper
  integral closed subscheme $Z\subset S$ of dimension $r$,
  \begin{displaymath}
    \deg\bigl(\kappa_r(f)\cap[Z]\bigr)\geq0.
  \end{displaymath}
\end{proposition}

\begin{proof}
  By base change, we may restrict the family to $Z$. Applying
  Chow's lemma and resolution, we obtain a proper birational morphism
  $h\colon T\to Z$ with $T$ smooth and projective. Let $f_T$ be the pulled-back
  family, with total space $X_T$, boundary $D_T$, and relative
  dimension $d$. By
  \cite[Theorem~2.13 and the proof of Corollary~2.14]
  {patakfalvi2017ampleness-cm}, the divisor $K_{f_T}+D_T$ is nef.
  Functoriality and the projection formula give
  \begin{displaymath}
    \deg\bigl(\kappa_r(f)\cap[Z]\bigr)
      =\deg_T\kappa_r(f_T)
      =\deg_{X_T}\bigl((K_{f_T}+D_T)^{d+r}\bigr)\geq0.
  \end{displaymath}
  The last intersection is nonnegative on each irreducible component
  of $X_T$.
\end{proof}

Let $f\colon(X,D)\to S$ be a family over an integral base. Choose a
finite surjective morphism $p\colon T\to S$, with $T$ normal and
integral, such that the irreducible components of the geometric generic
fiber are defined over $k(T)$. Such a cover is obtained by normalizing
$S$ in a suitable finite extension of $k(S)$. Normalize the pulled-back
family and write
\begin{displaymath}
  \bigsqcup_a(X_a,D_a+\Gamma_a)\longrightarrow (X,D)\times_S T,
  \qquad f_a\colon(X_a,D_a+\Gamma_a)\to T.
\end{displaymath}
By \cite[Lemma~3.2]{patakfalvi2017ampleness-cm}, the $f_a$ are
KSBA-stable families. Their generic fibers are geometrically integral
and log canonical. We call them the \emph{normalized component
families}, and set
\begin{displaymath}
  \var^\nu f:=\max_a\{\var f_a\}.
\end{displaymath}
The definition is independent of the cover. On a common finite normal
cover the component families agree up to permutation by uniqueness of
KSBA extensions, and finite base change preserves variation.
If the geometric components are already defined over $k(S)$, one may
take $T=S^\nu$. The finite cover is needed in general because monodromy
may permute the geometric components.

\begin{example}[Variation and normalized variation]
  \label{ex:variation-glued-curves}
  Glue two independently varying one-pointed families of stable curves
  over the factors of $S=B_1\times B_2$, where the $B_i$ are smooth
  projective integral curves and the generic fibers of the two families
  are smooth. The
  resulting family $f$ has
  \begin{displaymath}
    \var^\nu f=1,\qquad \var f=2.
  \end{displaymath}
  Let $\alpha$ and $\beta$ be the $\kappa_1$-classes of its normalized
  component families. They are pulled back from $B_1$ and $B_2$,
  respectively, and have positive degree on those curves by the
  ampleness of the log CM class. Thus
  $\alpha^2=\beta^2=0$ and $\deg_S(\alpha\beta)>0$.
  Normalization additivity gives $\kappa_1(f)=\alpha+\beta$ and
  $\kappa_2(f)=0$, since each component family comes from a curve.
  Consequently,
  \begin{displaymath}
    \kappa_2(f)=0,
    \qquad
    \deg_S\bigl(\kappa_1(f)^2\bigr)=2\deg_S(\alpha\beta)>0.
  \end{displaymath}
  This illustrates the distinction between normalized variation,
  which controls the vanishing of individual kappa classes in
  Theorem~\ref{thm:numerical-positivity}, and total variation,
  which controls the vanishing of powers of $\kappa_1$.
\end{example}

For a proper base $S$, an operational class $c\in A^r(S)_\bQ$ is
\emph{numerically trivial},
written $c\equiv0$, if $\deg(c\cap[Z])=0$ for every integral closed
subscheme $Z\subset S$ of dimension~$r$.

\begin{lemma}[Descent of numerical and Chow vanishing]
  \label{lem:vanishing-descent}
  Let $p\colon T\to S$ be a proper surjective morphism of proper
  integral schemes, and let $c\in A^r(S)_\bQ$.
  \begin{enumerate}
  \item One has $c=0 \Longleftrightarrow p^*c=0$ and 
    $c\equiv0 \Longleftrightarrow p^*c\equiv0.$
  \item Suppose $p^*c=\sum_a c_a$, where the classes
    $c_a\in A^r(T)_\bQ$ are numerically nonnegative and satisfy
    $c_a\equiv0\Rightarrow c_a=0$. Then
    \begin{math}
      c\equiv0 \Longleftrightarrow c=0.
    \end{math}
  \end{enumerate}
\end{lemma}

\begin{proof}
  After any base change $U\to S$, every integral cycle $Z\subset U$
  is dominated by an integral subscheme $W\subset T\times_S Z$
  that is proper and generically finite of some degree $e_Z>0$ over
  $Z$: take the closure of a closed point of the generic fiber.
  With $q\colon W\to U$ the induced morphism, the projection formula
  gives
  \begin{displaymath}
    e_Z\,c|_U\cap[Z]=q_*\bigl(c|_W\cap[W]\bigr).
  \end{displaymath}
  If $p^*c=0$, this is zero for every $U$ and $Z$, so $c=0$ with
  rational coefficients. By using the projection formula, we see that
  numerical triviality pulls back. For the converse, take $U=S$, $\dim Z=r$, and
  degrees in the displayed identity.

  For (2), a numerically trivial sum of nonnegative classes has every
  summand numerically trivial. Thus $c\equiv0$ implies $c_a=0$ for
  all $a$, and (1) gives $c=0$.
\end{proof}

\begin{theorem}[Kappa classes and normalized variation]
  \label{thm:numerical-positivity}
  Let $f\colon(X,D)\to S$ be a KSBA-stable family over a proper
  integral scheme of dimension $b$, and let $0\leq r\leq b$. Then
  \begin{displaymath}
    \kappa_r(f)=0\text{ in }A^r(S)_\bQ
    \quad\Longleftrightarrow\quad
    \kappa_r(f)\equiv0
    \quad\Longleftrightarrow\quad
    r>\var^\nu f.
  \end{displaymath}
  More precisely, for every integral $r$-dimensional subscheme
  $Z\subset S$,
  \begin{displaymath}
    \deg\bigl(\kappa_r(f)\cap[Z]\bigr)>0
    \quad\Longleftrightarrow\quad
    \var^\nu(f|_Z)=r.
  \end{displaymath}
  If $S$ is projective and $H$ is an ample Cartier divisor, then
  \begin{displaymath}
    \deg_S\bigl(H^{b-r}\kappa_r(f)\bigr)>0
    \quad\Longleftrightarrow\quad
    r\leq\var^\nu f.
  \end{displaymath}
\end{theorem}

\begin{proof}
  Let $g\colon(V,E)\to B$ be a KSBA family over a proper integral
  base of dimension $c$, with geometrically integral log-canonical
  generic fiber and $\var g=c$. Resolve the base as in
  Proposition~\ref{prop:kappa-nonnegative}. On the pulled-back family,
  the relative log canonical divisor is big and nef by
  \cite[Proposition~2.15]{patakfalvi2017ampleness-cm}. Hence
  $\deg_B\kappa_c(g)>0$.

  Apply this to the normalized component families over a finite cover
  $p\colon T\to S$ as above, of degree $e$. Normalization additivity
  and the projection formula give
  \begin{displaymath}
    e\deg_S\kappa_b(f)=\sum_a\deg_T\kappa_b(f_a).
  \end{displaymath}
  Every summand is nonnegative by
  Proposition~\ref{prop:kappa-nonnegative}. It vanishes if
  $\var f_a<b$, by Theorem~\ref{thm:vanishing}, and is positive if
  $\var f_a=b$, by the preceding paragraph. Thus the top degree is
  positive exactly when $\var^\nu f=b$. Applying this result to
  $f|_Z$ proves the assertion about $r$-dimensional subschemes.

  Put $w:=\var^\nu f$. If $r>w$, then
  \begin{displaymath}
    p^*\kappa_r(f)=\sum_a\kappa_r(f_a)=0
    \qquad\text{in }A^r(T)_\bQ
  \end{displaymath}
  using normalization additivity and Theorem~\ref{thm:vanishing}.
  Lemma~\ref{lem:vanishing-descent}(1) gives $\kappa_r(f)=~0$.

  Conversely, let $r\leq w$. For $r=0$, use $\kappa_0(f)=v>0$.
  For $r>0$, choose $a$ with $\var f_a\geq r$, and a smooth
  projective variety $T'$ with a proper birational morphism to $T$.
  A sufficiently general complete intersection $W\subset T'$ of
  dimension $r$ has image of dimension $r$ under the moduli map of
  $f_a$. It also meets the open locus where the component fibers are
  geometrically integral and log canonical, and where $T'\to S$ is
  finite. The top-degree result gives
  $\deg_W\kappa_r(f_{a,T'}|_W)>0$. Pulling normalization additivity back
  to $W$ and using nonnegativity of the other summands gives a positive
  degree for the pullback of $\kappa_r(f)$. Let $Z$ be the image of
  $W\to T'\to T\to S$. The map $W\to Z$ is generically finite.
  The projection formula therefore gives
  $\deg(\kappa_r(f)\cap[Z])>0$.

  Finally, suppose $S$ is projective. Since $p$ is finite, $p^*H$ is
  ample on $T$. For $0<r\leq w$, take a sufficiently general
  $r$-dimensional complete intersection
  $W\in|mp^*H|^{b-r}$, with $m$ sufficiently large; take $W=T$ if
  $r=b$. As above, the restriction of a component family with variation
  at least $r$ has positive top kappa degree on $W$, and the other
  summands are nonnegative. Thus
  \begin{displaymath}
    0<\deg_W\kappa_r(f_T|_W)
      =m^{b-r}e\deg_S\bigl(H^{b-r}\kappa_r(f)\bigr).
  \end{displaymath}
  The case $r=0$ follows from $\kappa_0(f)=v$, and for $r>w$ the
  intersection vanishes by the Chow vanishing just proved.
\end{proof}

For KSBA families $g_i$, indexed by a nonempty set $I$, let
$g_I$ be their fiber product over the common base, with boundary the
sum of the pulled-back boundaries. We call $\var g_I$ their
\emph{joint variation}. On a projective base, it is also the dimension
of the image of the tuple of coarse moduli maps. Indeed, the product
formula expresses $\kappa_1(g_I)$ as a positive linear combination of
the $\kappa_1(g_i)$. By Section~\ref{sec:kappa01}, these are pullbacks
of ample CM classes. The numerical dimension of the sum is therefore
both $\var g_I$ and the dimension of the tuple image. The same
description holds on a proper base after a projective birational
modification. In particular, joint variation is monotone under
enlarging $I$, and repetitions of a single family have its variation.

For a nef divisor $L$ on a projective variety $V$, write
$\operatorname{nd}_V(L)$ for its numerical dimension.

\begin{lemma}[Numerical dimension of the relative log canonical divisor]
  \label{lem:log-canonical-numerical-dimension}
  Let $f\colon(X,D)\to S$ be a KSBA-stable family of relative
  dimension $d$ over a smooth projective integral base of dimension
  $b$, with geometrically integral log-canonical generic fiber. Then
  $\Lambda:=K_f+D$ is nef and
  \begin{displaymath}
    \operatorname{nd}_X(\Lambda)=d+\var f.
  \end{displaymath}
\end{lemma}

\begin{proof}
  The divisor $\Lambda$ is nef by
  \cite[Theorem~2.13 and the proof of Corollary~2.14]
  {patakfalvi2017ampleness-cm}. Since the generic fiber is geometrically
  integral and log canonical, $v:=\var f=\var^\nu f$. Choose an ample
  Cartier divisor $A$ on $S$ sufficiently positive that
  $C:=\Lambda+f^*A$ is ample. For $0\leq k\leq b$, expansion and the
  projection formula give
  \begin{displaymath}
    \deg_X\bigl(\Lambda^{d+k}C^{b-k}\bigr)
      =\sum_{j=0}^{b-k}\binom{b-k}{j}
        \deg_S\bigl(\kappa_{k+j}(f)A^{b-k-j}\bigr).
  \end{displaymath}
  By Theorem~\ref{thm:numerical-positivity}, all summands are
  nonnegative, the term $j=0$ is positive if $k\leq v$, and all
  summands vanish if $k>v$. Thus
  $\operatorname{nd}_X(\Lambda)=d+v$.
\end{proof}

\begin{lemma}[Positivity of mixed nef intersections]
  \label{lem:nef-mixed-positivity}
  Let $E_1,\dotsc,E_m$ be nef $\bQ$-divisors on an integral
  projective $n$-fold $V$, and let $\alpha_1,\dotsc,\alpha_m\geq0$
  be integers with $\sum_j\alpha_j=n$. Then
  \begin{equation}\label{eq:nef-mixed-positivity}
    \deg_V\Bigl(\prod_jE_j^{\alpha_j}\Bigr)>0
    \quad\Longleftrightarrow\quad
    \sum_{j\in J}\alpha_j\leq
      \operatorname{nd}_V\Bigl(\sum_{j\in J}E_j\Bigr)
    \quad\text{for all }J.
  \end{equation}
  Here $J$ ranges over the nonempty subsets of $\{1,\dotsc,m\}$.
\end{lemma}

For semiample divisors, this follows from the multidegree criterion
\cite[Theorem~A]{castillo2020when-are}; see also
\cite[Theorems~2.14 and~2.19]{kaveh2016complete-intersections}.
For nef classes on a compact K\"ahler variety, the criterion follows
from Hu--Xiao \cite[Corollary~4.7]{hu2022hard-lefschetz} applied to
$\alpha_j$ copies of each $E_j$.
Relative log canonical divisors need not be semiample, even for
stable pointed curves
\cite[Theorem~3.0]{keel1999basepoint-freeness}.

\begin{proof}
  The case $n=0$ is immediate. Choose an ample divisor $A$ on $V$
  and put
  \begin{displaymath}
    P(t_0,\dotsc,t_m)
      :=\deg_V\Bigl(\bigl(t_0A+\sum_{j=1}^m t_jE_j\bigr)^n\Bigr).
  \end{displaymath}
  The coefficients of $P$ are nonnegative and its support
  $\mathcal B$ contains $(n,0,\dotsc,0)$. By Br\"and\'en--Huh
  \cite[Corollary~4.7]{branden2020lorentzian}, $\mathcal B$ is
  $M$-convex, i.e. the set of integer bases of a polymatroid. Thus,
  writing $\beta(J):=\sum_{j\in J}\beta_j$ and
  $\rho(J):=\max_{\beta\in\mathcal B}\beta(J)$, a nonnegative
  integer vector $\beta$ of total degree $n$ belongs to $\mathcal B$
  exactly when $\beta(J)\leq\rho(J)$ for all
  $J\subseteq\{0,\dotsc,m\}$; see
  \cite[Introduction and \S2.3]{castillo2020when-are}.
  If $0\in J$, then $\rho(J)=n$, since $t_0^n$ occurs in $P$.
  For $0\notin J$, setting $t_j=u$ on $J$ and $t_j=v$ otherwise gives
  \begin{displaymath}
    \deg_V\Bigl(\bigl(u\sum_{j\in J}E_j
      +v(A+\sum_{j\notin J,\,j\geq1}E_j)\bigr)^n\Bigr).
  \end{displaymath}
  Since the divisor multiplying $v$ is ample, the degree in $u$ is
  $\operatorname{nd}_V(\sum_{j\in J}E_j)$. By nonnegativity of the
  coefficients, it equals $\rho(J)$. Apply these rank inequalities
  to $(0,\alpha_1,\dotsc,\alpha_m)$. The coefficient of
  $t_1^{\alpha_1}\dotsm t_m^{\alpha_m}$ in $P$ is
  \begin{displaymath}
    \frac{n!}{\prod_j\alpha_j!}
      \deg_V\Bigl(\prod_jE_j^{\alpha_j}\Bigr),
  \end{displaymath}
  which proves \eqref{eq:nef-mixed-positivity}.
\end{proof}

\begin{theorem}[Positivity of kappa monomials]
  \label{thm:monomial-positivity}
  \label{prop:top-kappa-positivity}
  \label{prop:kappa_r=0}
  Let $f_i\colon(X_i,D_i)\to S$, $1\leq i\leq m$, be KSBA-stable
  families of relative dimensions $d_i$ over a proper integral base of
  dimension $b$. Let $r_i\geq0$ and $r:=\sum_i r_i\leq b$.
  If $r=b$, then
  \begin{displaymath}
    \deg_S\Bigl(\prod_i\kappa_{r_i}(f_i)\Bigr)\geq0.
  \end{displaymath}
  If $S$ is projective and $H$ is an ample Cartier divisor, then
  \begin{displaymath}
    \deg_S\Bigl(H^{b-r}\prod_i\kappa_{r_i}(f_i)\Bigr)\geq0.
  \end{displaymath}

  Suppose first that each $f_i$ has geometrically integral log-canonical
  generic fiber. In either case, strict positivity holds if and only if
  \begin{equation}\label{eq:joint-variation-criterion}
    \sum_{i\in I}r_i\leq\var\Bigl(\prod_{i\in I}f_i\Bigr)
    \qquad\text{for every nonempty }I\subseteq\{1,\dotsc,m\}.
  \end{equation}
  If all these families equal a single family $f$, it reduces to
  $r\leq\var f$. Thus every top-degree kappa monomial of such a
  family is positive exactly when the variation is maximal.

  In general, choose a common finite normal cover $p\colon T\to S$
  splitting the geometric generic components of all the families, and
  denote their normalized component families by $f_{i,a}$.
  For a tuple $\ba=(a_1,\dotsc,a_m)$, put
  \begin{displaymath}
    v_{\ba}(I):=\var\Bigl(\prod_{i\in I}f_{i,a_i}\Bigr),
    \qquad \varnothing\ne I\subseteq\{1,\dotsc,m\}.
  \end{displaymath}
  Strict positivity holds if and only if there is a tuple $\ba$ such
  that
  \begin{equation}\label{eq:normalized-joint-variation-criterion}
    \sum_{i\in I}r_i\leq v_{\ba}(I)
    \qquad\text{for every nonempty }I\subseteq\{1,\dotsc,m\}.
  \end{equation}
\end{theorem}

\begin{proof}
  Write $e:=\deg p$. Normalization additivity and the projection
  formula give
  \begin{equation}\label{eq:normalization-monomial-degree}
    e\deg_S\Bigl(H^{b-r}\prod_i\kappa_{r_i}(f_i)\Bigr)
      =\sum_{a_1,\dotsc,a_m}
        \deg_T\Bigl((p^*H)^{b-r}
          \prod_i\kappa_{r_i}(f_{i,a_i})\Bigr).
  \end{equation}
  For $r=b$, omit the factor involving $H$; the formula then applies
  to a proper base.

  Fix a tuple on the right. Choose a smooth projective variety
  $T'$ with a proper birational morphism $h\colon T'\to T$.
  Intersection numbers and joint variations are unchanged by this
  base change. When the factor involving $H$ occurs, its pullback
  $B=h^*p^*H$ is big and nef. We prove the assertion for any big and
  nef divisor $B$ on $T'$; for $r=b$, choose any ample $B$.

  Write $S$ for $T'$ and $f_i$ for the chosen component families.
  These have geometrically integral log-canonical generic
  fiber. Put $\Lambda_i:=K_{f_i}+D_i$. These divisors are nef by
  Lemma~\ref{lem:log-canonical-numerical-dimension}.

  Set $Y:=X_1\times_S\cdots\times_SX_m$, with projections $p_i$ and
  structure morphism $g\colon Y\to S$. Put $q_i:=d_i+r_i$ and
  $N_i:=p_i^*\Lambda_i$. Flatness and geometric integrality of the
  generic fibers imply that $Y$ and all partial fiber products $X_I$
  are integral. Repeated proper--flat base change and the projection
  formula give
  \begin{displaymath}
    \deg_S\Bigl(B^{b-r}\prod_i\kappa_{r_i}(f_i)\Bigr)
      =\deg_Y\Bigl((g^*B)^{b-r}\prod_i N_i^{q_i}\Bigr).
  \end{displaymath}
  All divisors on the right are nef, and their exponents sum to
  $\dim Y=b+\sum_i d_i$. Hence the degree is nonnegative.

  For nonempty $I$, put $d_I:=\sum_{i\in I}d_i$,
  $v_I:=\var f_I$, and $\Lambda_I:=K_{f_I}+D_I$.
  Lemma~\ref{lem:log-canonical-numerical-dimension} gives
  $\operatorname{nd}_{X_I}(\Lambda_I)=d_I+v_I$.
  Numerical dimension of a nef divisor is unchanged by surjective
  pullback: this follows from the projection formula after intersecting
  with enough general ample divisors on the source. Thus, for the
  projection $Y\to X_I$,
  \begin{equation}\label{eq:joint-variation-nd-on-Y}
    \operatorname{nd}_Y\Bigl(\sum_{i\in I}N_i\Bigr)=d_I+v_I.
  \end{equation}
  Let $u_I>0$ be the volume of a fiber of $f_I$. Since $B$ is big
  and nef, $B^b>0$, and
  \begin{displaymath}
    (\Lambda_I+f_I^*B)^{d_I+b}
      \geq \binom{d_I+b}{d_I}
        \Lambda_I^{d_I}(f_I^*B)^b
      =\binom{d_I+b}{d_I}u_I B^b>0.
  \end{displaymath}
  Here all terms of the expansion are nonnegative. Thus
  $\Lambda_I+f_I^*B$ is big and nef, and
  \begin{equation}\label{eq:big-nef-base-nd}
    \operatorname{nd}_Y\Bigl(\sum_{i\in I}N_i+g^*B\Bigr)=d_I+b,
    \qquad \operatorname{nd}_Y(g^*B)=b.
  \end{equation}
  Apply Lemma~\ref{lem:nef-mixed-positivity} on $Y$ to the divisors
  $g^*B,N_1,\dotsc,N_m$ with exponents $b-r,q_1,\dotsc,q_m$.
  The inequality for $g^*B$ alone is $b-r\leq b$. For a subset
  consisting of $g^*B$ and the $N_i$ with $i\in I$, it is automatic
  by \eqref{eq:big-nef-base-nd}, since
  \begin{displaymath}
    \sum_{i\in I}(d_i+r_i)+b-r\leq d_I+b,
  \end{displaymath}
  whereas the inequality for the $N_i$ with $i\in I$ is, by
  \eqref{eq:joint-variation-nd-on-Y},
  \begin{displaymath}
    d_I+v_I\geq\sum_{i\in I}(d_i+r_i),
  \end{displaymath}
  which is exactly \eqref{eq:joint-variation-criterion}.

  Summing the nonnegative degrees in
  \eqref{eq:normalization-monomial-degree} proves the assertion.
  For repeated copies of a family with geometrically integral
  log-canonical generic fiber, all joint variations equal $\var f$,
  so the condition is $r\leq\var f$.

\end{proof}

\begin{corollary}
  \label{cor:monomial-numerical-positivity}
  Let $f\colon(X,D)\to S$ be a KSBA-stable family over a proper
  integral base, and put $\Theta:=\prod_i\kappa_{r_i}(f)$ and
  $r:=\sum_i r_i$. Then $\Theta$ is nonnegative on effective $r$-cycles.
  Moreover,
  \begin{displaymath}
    \Theta=0\text{ in }A^r(S)_\bQ
    \quad\Longleftrightarrow\quad \Theta\equiv0.
  \end{displaymath}
  It is numerically nontrivial if and only if, for some integral
  $r$-dimensional $Z\subset S$, the normalized component families of
  $f|_Z$ admit a tuple satisfying
  \eqref{eq:normalized-joint-variation-criterion} over a finite cover
  of $Z$. In particular, $\Theta=0$ if $r>\var f$ or if some
  $r_i>\var^\nu f$.
\end{corollary}

\begin{proof}
  Numerical nonnegativity and the criterion on $r$-dimensional
  subschemes follow by applying Theorem~\ref{thm:monomial-positivity}
  to $f|_Z$. The last assertion follows from
  Theorems~\ref{thm:vanishing} and~\ref{thm:numerical-positivity}.

  For $\Theta\equiv0\Rightarrow\Theta=0$, the cases $r=0$ and $r>\dim S$
  are immediate. By using normalization additivity, nonnegativity in
  Theorem~\ref{thm:monomial-positivity}, and
  Lemma~\ref{lem:vanishing-descent}, it suffices to treat
  $N:=\prod_i\kappa_{r_i}(g_i)$ over a smooth projective base, where
  each $g_i$ has geometrically integral log-canonical generic fiber.
  If $N\equiv0$, Theorem~\ref{thm:monomial-positivity} gives a
  nonempty subset $I$ with
  $\sum_{i\in I}r_i>\var(\prod_{i\in I}g_i)$.
  By Proposition~\ref{prop:coarse-descent}, the corresponding partial
  product is pulled back from the joint coarse moduli image, whose
  dimension is this variation. It vanishes there by dimension, hence
  $N=0$. The converse is immediate.
\end{proof}

\begin{remark}
  For a fixed tuple, the inequalities $r_i\leq\var f_{a_i}$ alone do
  not suffice for positivity of its contribution. In
  Example~\ref{ex:variation-glued-curves}, taking the
  first component twice satisfies the singleton inequalities but gives
  $\alpha^2=0$. Taking the two different components gives
  $\deg_S(\alpha\beta)>0$; their joint variation is $2$.
\end{remark}

\begin{remark}[Khovanskii--Teissier inequalities]
  Let $f\colon(X,D)\to S$ be a family of relative dimension $d$
  over a projective integral base of dimension $b$. Choose a finite
  cover $p\colon T\to S$ of degree $e$ and normalized component
  families $f_a$ as above. For an ample Cartier divisor $H$ on $S$, set
  \begin{displaymath}
    s_k(H):=
    \deg\bigl(H^{b-k}\cdot\kappa_k(f)\cap[S]\bigr),
    \qquad 0\le k\le b.
  \end{displaymath}
  Put $H_T=p^*H$ and
  \begin{displaymath}
    \Lambda_a:=K_{f_a}+D_a+\Gamma_a.
  \end{displaymath}
  By using the projection formula and normalization additivity,
  we obtain $s_k(H)=e^{-1}\sum_a s_{a,k}(H)$, where
  \begin{displaymath}
    s_{a,k}(H):=
    \deg_{X_a}\bigl(
      (f_a^*H_T)^{b-k}\Lambda_a^{d+k}
    \bigr).
  \end{displaymath}
  The divisors $f_a^*H_T$ and $\Lambda_a$ are nef, as in
  Proposition~\ref{prop:kappa-nonnegative}. The Khovanskii--Teissier
  inequalities
  \cite{khovanskii1979geometry-convex-polyhedra,
    teissier1979hodge-index}, in the form of
  \cite[Example~1.6.4]{lazarsfeld2004positivityI}, give
  \begin{displaymath}
    s_{a,k}(H)^2
      \ge s_{a,k-1}(H)s_{a,k+1}(H)
    \qquad\text{for }1\le k\le b-1,
  \end{displaymath}
  and, by applying the inequalities to ample perturbations and passing
  to the limit,
  \begin{displaymath}
    s_{a,i}(H)^j
      \ge s_{a,0}(H)^{j-i}s_{a,j}(H)^i
    \qquad\text{for }0<i<j\le b.
  \end{displaymath}
  Here
  \begin{displaymath}
    s_{a,0}(H)=v_a\,(H_T)^b>0,
  \end{displaymath}
  where $v_a$ is the volume of a fiber of~$f_a$. All the numbers
  $s_{a,k}(H)$ are nonnegative. Therefore, if $s_i(H)=0$, then
  $s_{a,i}(H)=0$ for every~$a$, and the second inequality forces
  $s_{a,j}(H)=0$ for all $j>i$. It follows that $s_j(H)=0$ for
  every $j>i$.

  A sum of log-concave sequences need not be log-concave. If the
  normalized pulled-back family has a single component, then $s_k(H)$
  is log-concave, its positive terms form an initial interval, and the
  ratios $s_k(H)/s_{k-1}(H)$ are nonincreasing wherever defined.
\end{remark}

\subsection{Stacky operations on coarse moduli spaces}
\label{sec:stacky-coarse}

Every line bundle on a separated DM stack has a positive tensor
power that descends to its coarse moduli space
\cite[Lemma~3.2]{kresch2004on-coverings}; this gives rational
surjectivity of the ordinary pullback on Picard groups. To distinguish
these pullbacks from the Chow-theoretic comparison maps recalled below,
write them as
\begin{displaymath}
  \pi_{\Pic}^*\colon\Pic(M)_\bQ\longrightarrow\Pic(\cM)_\bQ,
  \qquad
  (\pi')_{\Pic}^*\colon\Pic(X)_\bQ\longrightarrow\Pic(\cX)_\bQ.
\end{displaymath}
For injectivity, note that in characteristic zero the stacks are tame and that
\begin{displaymath}
  \pi_*\cO_{\cM}=\cO_{M},\qquad
  \pi'_*\cO_{\cX}=\cO_{X}.
\end{displaymath}
If $F\in\Pic(M)$ and
$\pi_{\Pic}^*F\simeq\cO_{\cM}$, the projection formula gives
\begin{displaymath}
  F\simeq\pi_*\pi_{\Pic}^*F
    \simeq\pi_*\cO_{\cM}
    \simeq\cO_{M},
\end{displaymath}
and the same argument applies to $\pi'$. Consequently,
\begin{displaymath}
  \pi_{\Pic}^*\colon\Pic(M)_\bQ\isoto\Pic(\cM)_\bQ,
  \qquad
  (\pi')_{\Pic}^*\colon\Pic(X)_\bQ\isoto\Pic(\cX)_\bQ.
\end{displaymath}

Recall from Section~\ref{sec:stack-to-coarse} that the coarse moduli
maps induce the Chow-theoretic comparison isomorphisms
\begin{align*}
  \pi_*&\colon A_*(\cM)_\bQ \isoto A_*(M)_\bQ,
  &
  \pi'_*&\colon A_*(\cX)_\bQ \isoto A_*(X)_\bQ,\\
  \pi^*&\colon A^*(M)_\bQ \isoto A^*(\cM)_\bQ,
  &
  \pi'{}^*&\colon A^*(X)_\bQ \isoto A^*(\cX)_\bQ.
\end{align*}
In the remainder of this subsection, the unadorned symbols
$\pi_*$, $\pi'_*$, $\pi^*$, and $\pi'{}^*$ always denote these
Chow comparison maps. The Picard pullbacks are denoted
$\pi_{\Pic}^*$ and $(\pi')_{\Pic}^*$ as above; after applying the first
Chern class, they are compatible with the degree-one operational
pullbacks $\pi^*$ and $\pi'{}^*$.

Since $\bar f$ need not be flat, neither an ordinary flat pullback in
Chow homology nor an ordinary flat Gysin pushforward in operational
Chow cohomology is generally available. We instead make the following
definition.
\begin{definition}
  The \emph{stacky pullback} and the \emph{stacky Gysin pushforward}
  associated with $\bar f$ are, respectively,
  \begin{align*}
    \stpull\colon A_k(M)_\bQ
      &\longrightarrow A_{k+d}(X)_\bQ,
    &
    \stpull(\alpha)
    &
    = \bigl(\pi'_*\! \circ \ff^* \circ \pi_*^{-1} \bigr)(\alpha),\\
    \stpush\colon A^p(X)_\bQ
      &\longrightarrow A^{p-d}(M)_\bQ,
    &
    \stpush(c)
      &= \bigl(\pi^*{}^{-1} \circ \ff_* \circ \pi'{}^*\bigr) (c).
  \end{align*} 
\end{definition}

Let
\begin{displaymath}
  \overline{\cL}_N
    :=\bigl((\pi')_{\Pic}^*\bigr)^{-1}[\cL_N]
      \in\Pic(X)_\bQ,
\end{displaymath}
and set
\begin{displaymath}
  \bar\ell_N:=c_1(\overline{\cL}_N),
  \qquad
  \overline{K_\ff+\cD}:=\frac1N\bar\ell_N\in A^1(X)_\bQ.
\end{displaymath}
Here the overline denotes the descended class. By the compatibility
of Picard and operational pullbacks, $\pi'{}^*\bar\ell_N=\ell_N$,
and hence $\pi'{}^*(\overline{K_\ff+\cD})=K_\ff+\cD$.
The transported operations recover the descended kappa classes by the formula
\begin{equation}\label{eq:coarse-kappa-formula}
  \bar\kappa_r
    =\frac1{N^{r+d}}\,\stpush\bigl(\bar\ell_N^{r+d}\bigr)
    \qquad\text{in }A^r(M)_\bQ.
\end{equation}
In particular, for every $\alpha\in A_*(M)_\bQ$,
\begin{displaymath}
  \bar\kappa_r\cap\alpha
    =\frac1{N^{r+d}}\,\bar f_*\bigl(
      \bar\ell_N^{r+d}\cap\stpull\alpha
    \bigr).
\end{displaymath}
Indeed, pulling \eqref{eq:coarse-kappa-formula} back by $\pi$ gives
Definition~\ref{def:kappa}, and the second identity is the defining
bivariant action transported through the coarse-space comparison
isomorphisms.

For every morphism $S\to\cM$, define the corresponding operations for
the pulled-back family by the same transport through the coarse-space
comparison isomorphisms. With this base-change-indexed definition,
compatibility with every further base change $T\to S$ follows directly
from the bivariant compatibility of flat pullback and Gysin pushforward
for the Cartesian pullback of $\ff$. Thus the stacky operations commute
with base changes that arise over $\cM$; no Cartesian assertion about
diagram~\eqref{eq:coarse-moduli} is being made. 

\begin{lemma}
  Let $\cV$ be an integral substack of $\cM$, and write the flat
  pullback of its fundamental cycle as
  \begin{displaymath}
    \ff^*[\cV]=\sum_i[\cW_i],
  \end{displaymath}
  with integral $\cW_i$. Let $V\subset M$ and $W_i\subset X$
  be the corresponding reduced subvarieties. Then
  \begin{displaymath}
    \stpull [V] = \sum_i \frac{e_V}{e_{W_i}} [W_i],
  \end{displaymath}
  where $e_V$ is the order of the automorphism group of a generic
  point of $\cV$, and similarly for $e_{W_i}$.
\end{lemma}
\begin{proof}
  This is immediate from $\pi_*[\cV] = \frac1{e_V}[V]$ and
  $\pi'_*[\cW_i] = \frac1{e_{W_i}}[W_i]$.
\end{proof}

The coefficients of $[\cW_i]$ in $\ff^*[\cV]$ are all equal to~$1$
because the KSBA family~$\ff$ has reduced fibers. Since $\ff$ is
representable, the stabilizer of a generic point of $\cW_i$ injects
into the stabilizer of its image in $\cV$. Hence $e_{W_i}$ divides
$e_V$, and the coefficients ${e_V}/{e_{W_i}}$ in the preceding formula
are integers.

\subsection{Chambers and wall crossing}
\label{sec:wall-crossing}

We use the wall-crossing results of
\cite{ascher2023wall-crossing,meng2023mmp-for}, in particular
\cite[Theorems~1.1, 1.2, and 6.2]{meng2023mmp-for}.
Let
\begin{displaymath}
  (X,\textstyle\sum_{s=1}^n \hat a_sD_s)\longrightarrow B
\end{displaymath}
be a KSBA-stable family, where $B$ is normal and \emph{not necessarily
proper}, and assume that the generic fiber is normal. A coefficient vector
$\ba=(a_1,\dotsc,a_n)$ is called \emph{admissible} if
$0<a_s\le \hat a_s$ for every $s$ and the divisor
\begin{displaymath}
  K_{X/B} + \ba D := K_{X/B}+\sum_s a_sD_s
\end{displaymath}
is big over $B$. An admissible rational polytope is a rational polytope
all of whose points are admissible. The volume generally varies with
$\ba$, and we write $\mathcal K_{\ba}$ for the corresponding KSBA
moduli space.

For an admissible $\ba$, let $(Z_{\ba},\ba\Delta_{\ba})\to B$ be the
stable log canonical model, which exists by
\cite[Theorem~1.1]{meng2023mmp-for}. Let
$\Phi_{\ba}\colon B\to\mathcal K_{\ba}$ be its moduli map, and let
$\cM_{\ba}$ be the reduced closure of $\Phi_{\ba}(B)$. For an
admissible rational polytope $\cP$,
\cite[Theorem~1.2]{meng2023mmp-for} gives a finite rational polyhedral
decomposition
\begin{displaymath}
  \cP = \bigcup_{i=1}^p \cP_i
\end{displaymath}
with the following properties:
\begin{enumerate}
\item Inside an open chamber $\cP_i^0$ the underlying marked family
  \begin{displaymath}
    (\cX_\ba; \cD_{\ba,1}, \dotsc, \cD_{\ba,n})
  \end{displaymath}
  is independent of $\ba$ and the reduced closures $\cM_{\ba}$
  are canonically isomorphic. We write $\cM_i$ for $\cM_\ba$ when
  $\ba\in \cP_i^0$.
\item Suppose that $\cP_j$ is a face of $\cP_i$. Then there exists a
  commutative diagram
  \begin{equation}\label{eq:wall-crossing}
    \begin{tikzcd}
      (\cX_i; \cD_{i,1}, \dotsc, \cD_{i,n})
        \arrow[r] \arrow[d, "\ff_i"']
      & (\cX_j; \cD_{j,1}, \dotsc, \cD_{j,n})
        \arrow[d, "\ff_j"]
      \\
      \cM_i \arrow[r, "\alpha_{ij}"]
      & \cM_j
    \end{tikzcd}
  \end{equation}
  with wall-crossing morphisms between the reduced closures $\cM_i$ and
  $\cM_j$ and between their universal families.
\end{enumerate}

\begin{remark}
  The moduli loci and their reduced or normalized closures are
  Deligne--Mumford stacks, and all reduction and wall-crossing
  morphisms are morphisms of stacks. We use the term ``moduli space''
  in this sense. This convention is explicit in
  \cite{ascher2023wall-crossing}; in \cite{meng2023mmp-for} it is
  implicit, but the constructions apply to the moduli stacks.
\end{remark}

For $\ba=(a_1,\dotsc,a_n)\in\cP_i\cap\bQ^n$, put
$\ba\cD_i:=\sum_{s=1}^n a_s\cD_{i,s}$ and
\begin{displaymath}
  \Lambda_i(\ba)
    :=\cO_{\cX_i}\bigl(K_{\ff_i}+\ba\cD_i\bigr)
      \in\Pic(\cX_i)_\bQ.
\end{displaymath}
Define
\begin{displaymath}
  \kappa_r(\ff_i,\ba)
    :=(\ff_i)_*\bigl((K_{\ff_i}+\ba\cD_i)^{r+d}\bigr)
      \in A^r(\cM_i)_\bQ.
\end{displaymath}
For $\ba\in\cP_i^0$, this is the ordinary kappa class; on
$\cP_i\setminus\cP_i^0$ it denotes its polynomial extension on the
fixed chamber model.

\begin{proposition}[Chamberwise polynomiality]
  As the rational coefficient vector $\ba$ varies in the closed
  polytope $\cP_i$, the classes
  \begin{displaymath}
    \kappa_r(\ff_i,\ba)\in A^r(\cM_i)_\bQ
  \end{displaymath}
  depend polynomially on $\ba$, with total degree at most $r+d$.
\end{proposition}

\begin{proof}
  For rational points $\ba,\ba'\in\cP_i$, the divisors
  $K_{\ff_i}+\ba\cD_i$ and $K_{\ff_i}+\ba'\cD_i$ are
  $\bQ$-Cartier, so their difference $(\ba-\ba')\cD_i$ is
  $\bQ$-Cartier. Thus $\Lambda_i(\ba)$ depends affinely on $\ba$
  in $\Pic(\cX_i)_\bQ$, and so does its first Chern class in
  $A^1(\cX_i)_\bQ$. Its $(r+d)$-th power is polynomial of total
  degree at most $r+d$. Applying $(\ff_i)_*$ proves the assertion.
\end{proof}

\begin{theorem}[Compatibility with wall crossing]
  \label{thm:operational-wall-crossing}
  Suppose that $\cP_j$ is a face of $\cP_i$. Then, for every
  $\ba\in\cP_j\cap\bQ^n$, one has
  \begin{displaymath}
    \kappa_r(\ff_i,\ba)
      =\alpha_{ij}^*\kappa_r(\ff_j,\ba)
      \qquad\text{in }A^r(\cM_i)_\bQ.
  \end{displaymath}
\end{theorem}

\begin{proof}
  Set
  \begin{displaymath}
    \cX'_j:=\cX_j\times_{\cM_j}\cM_i,
  \end{displaymath}
  let $\ff'_j\colon\cX'_j\to\cM_i$ be the projection, and write
  $\cD'_{j,s}$ for the pullback of $\cD_{j,s}$ and
  $\ba\cD'_j:=\sum_{s=1}^n a_s\cD'_{j,s}$. The lifted wall-crossing
  map in \eqref{eq:wall-crossing} induces a morphism
  $\rho\colon\cX_i\to\cX'_j$ and a commutative diagram
  \begin{displaymath}
    \begin{tikzcd}
      \cX_i
        \arrow[r, "\rho"]
        \arrow[d, "\ff_i"']
      & \cX'_j
        \arrow[r]
        \arrow[d, "\ff'_j"']
      & \cX_j
        \arrow[d, "\ff_j"]
      \\
      \cM_i
        \arrow[r, equal]
      & \cM_i
        \arrow[r, "\alpha_{ij}"]
      & \cM_j.
    \end{tikzcd}
  \end{displaymath}
  Put
  \begin{displaymath}
    \Lambda'_j(\ba)
      :=\cO_{\cX'_j}
        \bigl(K_{\ff'_j}+\ba\cD'_j\bigr)
      \in\Pic(\cX'_j)_\bQ.
  \end{displaymath}

  Let $\ba^\circ\in\cP_j^0\cap\bQ^n$. By
  \cite[proof of Theorem~6.2(3)]{meng2023mmp-for}, the morphism
  $\rho\colon\cX_i\to\cX'_j$ is the relative ample-model morphism of
  $K_{\ff_i}+\ba^\circ\cD_i$. Hence, by
  \cite[proof of Lemma~4.9]{meng2023mmp-for},
  \begin{displaymath}
    [\Lambda_i(\ba^\circ)]
      =\rho^*[\Lambda'_j(\ba^\circ)]
      \qquad\text{in }\Pic(\cX_i)_\bQ.
  \end{displaymath}
  The marked families and $\rho$ are independent of
  $\ba^\circ\in\cP_j^0$. Since both sides depend affinely on $\ba$,
  this extends to
  \begin{displaymath}
    [\Lambda_i(\ba)]
      =\rho^*[\Lambda'_j(\ba)]
      \qquad\text{in }\Pic(\cX_i)_\bQ
  \end{displaymath}
  for every $\ba\in\cP_j\cap\bQ^n$.

  Let $Z\to\cM_i$ be an integral test scheme, with generic point
  $\eta$. The pulled-back families are flat with geometrically reduced
  fibers, so $\cX_{i,Z}$ and $\cX'_{j,Z}$ are reduced, and every
  irreducible component dominates $Z$. By \cite[proofs of
  Theorem~6.2(3) and Lemma~4.9]{meng2023mmp-for}, the geometric
  generic fiber of $\rho_Z$ is a surjective ample-model contraction
  with connected fibers. Since source and target are pure
  $d$-dimensional, it is generically finite over every target component;
  in characteristic zero, connectedness of the fibers forces degree
  one. Components contracted to lower dimension contribute zero to
  proper pushforward. Thus
  \begin{displaymath}
    (\rho_\eta)_*[\cX_{i,\eta}]=[\cX'_{j,\eta}],
  \end{displaymath}
  and, since all components dominate $Z$, the same equality holds on
  the total spaces:
  \begin{displaymath}
    (\rho_Z)_*[\cX_{i,Z}]=[\cX'_{j,Z}].
  \end{displaymath}
  Theorem~\ref{thm:crepant-functoriality}(2) and
  base-change functoriality give
  \begin{align*}
    \kappa_r(\ff_i,\ba)
      &=(\ff'_j)_*\bigl((K_{\ff'_j}+\ba\cD'_j)^{r+d}\bigr) \\
      &=\alpha_{ij}^*\kappa_r(\ff_j,\ba).
  \end{align*}
\end{proof}

\begin{corollary}[Piecewise-polynomial numerical pairings]
  \label{cor:pl-pairing}
  Let $T$ be a proper scheme, let $g\colon T\to\cM_i$ be a morphism,
  and let $\gamma\in A_t(T)_\bQ$. Let $r_1,\dotsc,r_m$ be nonnegative
  integers such that
  \begin{math}
    r_1+\dotsb+r_m=t.
  \end{math}
  Then the function on rational points
  \begin{displaymath}
    \ba\longmapsto
    \deg\Bigl(\prod_{a=1}^m
      g^*\kappa_{r_a}(\ff_i,\ba)\cap\gamma\Bigr)
  \end{displaymath}
  is polynomial on $\cP_i\cap\bQ^n$, of total degree at most
  $md+t$. If $\cP_j$ is a face of $\cP_i$, then on $\cP_j\cap\bQ^n$ it
  agrees with
  \begin{displaymath}
    \ba\longmapsto
    \deg\Bigl(\prod_{a=1}^m
      (\alpha_{ij}\circ g)^*\kappa_{r_a}(\ff_j,\ba)
      \cap\gamma\Bigr).
  \end{displaymath}
\end{corollary}

\begin{proof}
  Each class $\kappa_{r_a}(\ff_i,\ba)$ is polynomial in
  $\ba$ of total degree at most $r_a+d$, so their product has
  degree at most $\sum_a(r_a+d)=md+t$. Pullback, intersection with
  $\gamma$, and degree are linear. The equality on a face follows
  from Theorem~\ref{thm:operational-wall-crossing} applied to each
  factor.
\end{proof}

These identities are compatible with composition of wall-crossing
morphisms. After tensoring the Chow groups with $\bR$, the chamber
polynomials define classes for real coefficients. The identities on
faces still hold, since they hold on the dense set of rational points.

Let $\pi_i\colon\cM_i\to M_i$ be the coarse moduli map, and let
$\bar\alpha_{ij}\colon M_i\to M_j$ be the induced wall-crossing
morphism. Put
\begin{displaymath}
  \bar\kappa_r(\ff_i,\ba)
    :=(\pi_i^*)^{-1}\kappa_r(\ff_i,\ba).
\end{displaymath}
By naturality of pullback,
\begin{displaymath}
  \bar\kappa_r(\ff_i,\ba)
    =\bar\alpha_{ij}^*\bar\kappa_r(\ff_j,\ba)
\end{displaymath}
for every face inclusion $\cP_j\subset\cP_i$.

For $r=0$, the volume functions on the chambers glue to
a continuous piecewise-polynomial function of the coefficients. For
$r=1$, the logarithmic CM classes of the chamber models agree under
pullback across every wall.

\begin{proposition}[Monotonicity of kappa monomials]
  \label{prop:kappa-monomial-monotonicity}
  Let $g\colon S\to\cM_i$ be a morphism, and let
  $f\colon(X;D_1,\dotsc,D_n)\to S$ be the pulled-back chamber family
  of relative dimension $d$. Write
  $\kappa_q(f,\bb)=g^*\kappa_q(\ff_i,\bb)$ for its chamber classes.
  Let $\bb,\bb'\in\cP_i\cap\bQ^n$ with $\bb\ge\bb'$ coordinatewise,
  and let $r_1,\dotsc,r_m\ge0$ with $r=\sum_a r_a$. Then for every
  proper integral subscheme $Z\subset S$ of dimension $r$,
  \begin{displaymath}
    \deg\Bigl(\prod_{a=1}^m\kappa_{r_a}(f,\bb)\cap[Z]\Bigr)
      \ge
    \deg\Bigl(\prod_{a=1}^m\kappa_{r_a}(f,\bb')\cap[Z]\Bigr).
  \end{displaymath}
\end{proposition}

\begin{proof}
  By functoriality and the projection formula, we may replace $S$ by
  a smooth projective variety mapping properly and birationally to
  $Z$, as in Proposition~\ref{prop:kappa-nonnegative}. Put
  \begin{displaymath}
    L=K_f+\bb D,\qquad L'=K_f+\bb'D,\qquad
    E=L-L'=\sum_s(b_s-b_s')D_s.
  \end{displaymath}
  Both $L$ and $L'$ are nef: this holds for weights in $\cP_i^0$ by
  the positivity result used in that proposition, and for boundary
  weights by taking limits. The divisor $E$ is effective and
  $\bQ$-Cartier, being the difference of two log canonical
  $\bQ$-Cartier divisors; the individual $D_i$ need not be
  $\bQ$-Cartier.

  Let $Y=X\times_S\cdots\times_SX$ be the $m$-fold fiber product,
  with projections $p_a\colon Y\to X$, and set
  \begin{displaymath}
    N_a=d+r_a,\qquad L_a=p_a^*L,\qquad
    L'_a=p_a^*L',\qquad E_a=p_a^*E.
  \end{displaymath}
  The divisors $L_a,L'_a$ are nef, and $E_a$ is effective since
  $p_a$ is flat. Repeated proper--flat base change and the projection
  formula give
  \begin{displaymath}
    \deg_S\Bigl(\prod_a\kappa_{r_a}(f,\bb)\Bigr)
      =\deg_Y\Bigl(\prod_a L_a^{N_a}\Bigr),
  \end{displaymath}
  and similarly for $\bb'$. Replace the factors one at a time and,
  for $N_a>0$, use
  \begin{displaymath}
    L_a^{N_a}-(L'_a)^{N_a}
      =E_a\sum_{k=0}^{N_a-1}L_a^k(L'_a)^{N_a-1-k}.
  \end{displaymath}
  Each resulting degree is an intersection of nef divisors with the
  effective cycle $E_a\cap[Y]$, and is therefore nonnegative.
  This proves the inequality, also when $Y$ is reducible.
\end{proof}

The same conclusion holds for a coordinatewise nondecreasing segment
crossing several chambers, provided the chamber classes are pulled
back to a common proper test scheme by compatible morphisms.
Corollary~\ref{cor:pl-pairing} identifies the monomial pairings at the
walls, so the resulting piecewise-polynomial function is
nondecreasing.

\section{Psi classes and mixed classes}
\label{sec:psi-classes}

\subsection{Psi classes}
\label{sec:psi-curves}

Let $f\colon (X,\sum_{i=1}^nD_i)\to S$ be a family of $n$-pointed
stable curves of genus $g$, and let $s_i\colon S\to X$ be the sections
with images $D_i$. The sections lie in the smooth locus of $f$, so 
$D_i$ are relative Cartier divisors. The psi classes $\psi_i\in A^1(S)$ for
$1\le i\le n$ are defined by any of the following equivalent identities:
\begin{displaymath}
  \psi_i=c_1\bigl(s_i^*\omega_{X/S}\bigr)
    =c_1\bigl(\cO_{D_i}(-D_i)\bigr)
    =c_1(N_{D_i/X}^\vee) 
\end{displaymath}
The middle equality follows from adjunction. As elsewhere in this
paper, let us denote $c_1\bigl(\cO_X(D_i)\bigr)$ simply by $D_i$.  The
self-intersection and projection formulas give
\begin{math}
  \psi_i=-f_*(D_i^2).
\end{math}
Thus, for the universal family
$\ff\colon(\cX,\sum_{i=1}^n\cD_i)\to\ocM_{g,n}$,
\begin{displaymath}
  \psi_i=c_1\bigl(s_i^*\omega_{\cX/\ocM_{g,n}}\bigr)
    =-\ff_*(\cD_i^2)
  \qquad\text{in }A^1(\ocM_{g,n}).
\end{displaymath}

A direct generalization of this definition to the higher-dimensional
case is:
\begin{definition}\label{def:psi-classes}
  Let $f\colon (X,D = \sum_{i=1}^n a_i D_i)\to S$ be a family of
  $d$-dimensional KSBA-stable pairs. Suppose that for some
  $1\le i\le n$ the divisor $D_i$ is relatively $\bQ$-Cartier. In this
  situation, we define the generalized psi classes $\psi_{i,r}$ as the
  Gysin pushforward in operational Chow cohomology
  \begin{displaymath}
    \psi_{i,r} = (-1)^r f_* D_i^{d+r} \in A^r(S)_\bQ.
  \end{displaymath}
\end{definition}

In the case of $\ocM_{g,n}$, or more generally of Hassett's moduli
space of weighted stable curves, let $\ff_i = \ff|_{\cD_i} \colon \cD_i\to
\ocM_{g,\ba}$ be the induced isomorphisms. Then
\begin{displaymath}
  \psi_{i,r} = (\ff_i)_* (-\cD_i|_{\cD_i})^r = \psi_i^r
\end{displaymath}
because $\ff_i$ are isomorphisms. For $d\ge2$ and $r\ge2$,
$\psi_{i,r}$ need not equal $\psi_{i,1}^{\,r}$, as the example in
Section~\ref{sec:exampe-cy} shows.

The condition that $D_i$ is relatively $\bQ$-Cartier is not
unreasonable in many situations. For example, it holds for the moduli
spaces of log Calabi-Yau pairs $(X, \Delta + \epsilon D)$, $0<\epsilon
\ll 1$ for the divisor $D$. Now, let us treat the general case.

For the following construction, we impose an additional hypothesis on
the setup of Section~\ref{sec:wall-crossing}: the fixed original
marked family $(X;D_1,\dotsc,D_n)\to B$ is stable for every weight
in the chosen rational polytope $\cP$.
(Recall that $B$ need not be proper, e.g. it may only cover a dense
open subset of a moduli space.)
Let $\cN_\ba$ be the
normalization of the reduced moduli locus $\cM_\ba$. If
$\ba,\ba-\beps\in\cP$, with $\epsilon_j\ge0$ sufficiently small, the
point $\ba$ lies in the closure of the chamber containing
$\ba-\beps$. Thus \cite[Theorem~1.2]{meng2023mmp-for} gives a
canonical dominant wall-crossing morphism
\begin{displaymath}
  \cM_{\ba-\beps}\longrightarrow\cM_\ba.
\end{displaymath}
Passing to normalizations gives a morphism
\begin{displaymath}
  \beta\colon\cN_{\ba-\beps}\longrightarrow\cN_\ba.
\end{displaymath}
On the other hand, under the additional stability hypothesis above,
\cite[Theorem~1.3]{meng2023mmp-for} gives the canonical reduction
morphism in the opposite direction:
\begin{displaymath}
  r_{\ba,\ba-\beps}\colon\cN_\ba\longrightarrow\cN_{\ba-\beps}.
\end{displaymath}
On the common dense open locus of marked pairs stable at both
weights, these morphisms simply change the coefficients on the
same marked pair and act identically on its automorphisms. Their
compositions are therefore isomorphic to the identity on this open
locus. Normality and the finite diagonals of these separated DM
stacks extend these identifications globally. Thus $\beta$ and
$r_{\ba,\ba-\beps}$ are inverse canonical isomorphisms.
This identification is also proved in the more restricted setting of
\cite[Theorem~1.9]{ascher2023wall-crossing}. It need not hold in the
full generality of Section~\ref{sec:wall-crossing}, as
\cite[Example~7.2]{meng2023mmp-for} shows.

\begin{definition}\label{def:psi-general}
  Under the additional stability hypothesis above, take
  $\beps = (\epsilon_j)$ such that $\epsilon_j=0$ for $j\ne i$,
  $\epsilon_i>0$, and $\ba-\beps\in\cP$. Then for any
  $0<\epsilon_i\ll1$ the vector
  $\ba-\beps$ lies in the relative interior of a unique cell, and the
  underlying marked universal family over $\cN_{\ba-\beps}$ does not
  depend on $\beps$. The divisor $\cD_i$ is relatively $\bQ$-Cartier:
  subtracting the relative log canonical divisors at two distinct
  nearby rational weights on this coordinate ray gives a nonzero
  rational multiple of $\cD_i$. Thus, the psi classes
  $\psi_{i,r}$ are well defined on $\cN_{\ba-\beps}$ and we define
  them on $\cN_\ba$ using the identification
  $\cN_{\ba-\beps}=\cN_\ba$.
\end{definition}

\subsection{Numerical properties of psi classes}
\label{sec:psi-positivity}

It is well-known that on $\ocM_{g,n}$ the classes $\psi_i$ are
nef. Except for $(g,n)=(0,3)$ and $(1,1)$, taking $i=n$, regard
$\ocM_{g,n}$ as the universal curve
$\pi\colon\ocM_{g,n}\to\ocM_{g,n-1}$, and let $\Sigma$ be the sum of
the first $n-1$ sections. By Knudsen
\cite[Theorem~4.1(d)]{knudsen1983the-projectivity-of-the-moduli3}, the
$n$-th cotangent line bundle is $\omega_\pi(\Sigma)$, so
$\psi_n=c_1\bigl(\omega_\pi(\Sigma)\bigr)$.  It is well-known that
$\omega_\pi(\Sigma)$ is big and nef. On $\ocM_{0,3}$ the psi classes
are zero, and on $\ocM_{1,1}$ one has $\psi_1=\lambda$, which is ample.
However, by Keel
\cite[Theorem~3.0 and Corollary~3.1]{keel1999basepoint-freeness} the psi
line bundle is not semiample in characteristic~$0$ for $g\ge3$, or
for $g=2$ and $n\ge2$. The exceptional case $\psi_1$ on $\ocM_{2,1}$
is semiample; see \cite{blankers2020hyperelliptic-classes}.
If $g=0$ then this psi
line bundle \emph{is} semiample: it defines Kapranov's birational
contraction $\oM_{0,n}\to\bP^{n-3}$ as in
\cite{kapranov1993veronese-curves}.

The psi classes of Definition~\ref{def:psi-classes} on Hassett's
moduli spaces $\ocM_{g,\ba}$ were investigated in
\cite{alexeev2008moduli-of-weighted}.  Lemmas~5.2 and~5.5 in that
paper show that in this case $\psi_i$ may not be nef.  For example, if
$\ba=(1, a_2, \dotsc, a_n)$ is a vector such that $\sum_{i=2}^n a_i>1$
but for any $2\le j_0 \le n$ one has
$(\sum_{i=2}^n a_i)-a_{j_0} \le 1$ then $\oM_{0,\ba}\simeq\bP^{n-3}$,
$\psi_1 = \cO(1)$ and $\psi_i= \cO(-1)$ for $i\ge 2$. The second
example is the Losev-Manin space
$\oM_{0, (1,1,\epsilon, \dotsc, \epsilon)}$ on which $\psi_i=0$ for
$i\ge3$.  Consequently, there is no reason to expect a psi class
$\psi_{i,r}$ on a higher-dimensional KSBA space to be nef unless the
coefficient $a_i$ is $1$.

\subsection{Mixed classes}
\label{sec:psi-def-max-chamber}

More generally, we can define \emph{mixed kappa-psi} classes as Gysin
pushforwards $f_* (K_f + D)^{m_0} \prod D_i^{m_i}$. 
To start,
let $f\colon X\to S$ be a proper flat family of pure relative
dimension $d\ge1$. For rational line bundles
$\Lambda_1,\dotsc,\Lambda_{r+d}$ on $X$, put
\begin{displaymath}
  \mu(f;\Lambda_1,\dotsc,\Lambda_{r+d})
    :=f_*\Bigl(\prod_{i=1}^{r+d}c_1(\Lambda_i)\Bigr)
    \in A^r(S)_\bQ.
\end{displaymath}
This class is symmetric and multilinear. The ordinary kappa class is
the diagonal specialization with all $\Lambda_i$ equal to the relative
log canonical $\bQ$-line bundle. 

\begin{proposition}
  \label{prop:mixed-functoriality}
  The classes defined above have the following properties.
  \begin{enumerate}
  \item For every base change $g\colon T\to S$, with projection
    $g_X\colon X_T\to X$ and induced family $f_T\colon X_T\to T$,
    \begin{displaymath}
      g^*\mu(f;\Lambda_1,\dotsc,\Lambda_{r+d})
        =\mu(f_T;g_X^*\Lambda_1,\dotsc,g_X^*\Lambda_{r+d}).
    \end{displaymath}
  \item In the setting of
    Theorem~\ref{thm:crepant-functoriality}, suppose that
    $\Lambda_{Y,a}=\rho^*\Lambda_{X,a}$ for every $a$. Under the
    hypotheses of Theorem~\ref{thm:crepant-functoriality}(1),
    \begin{displaymath}
      \mu(f_Y;\Lambda_{Y,1},\dotsc,\Lambda_{Y,r+d})\cap[S]
        =e\,\mu(f_X;\Lambda_{X,1},\dotsc,\Lambda_{X,r+d})\cap[S].
    \end{displaymath}
    Under the hypotheses of
    Theorem~\ref{thm:crepant-functoriality}(2), the same
    equality holds in $A^r(S)_\bQ$ without capping with $[S]$.
  \item In the setting of
    Proposition~\ref{prop:normalization-additivity}, write
    $\nu_a\colon X_a\to X$ for the normalized components and
    $f_a=f\circ\nu_a$. Then
    \begin{displaymath}
      \mu(f;\Lambda_1,\dotsc,\Lambda_{r+d})
        =\sum_a\mu(f_a;\nu_a^*\Lambda_1,\dotsc,\nu_a^*\Lambda_{r+d}).
    \end{displaymath}
  \end{enumerate}
\end{proposition}

\begin{proof}
  The first assertion is base-change compatibility of the operational
  Gysin construction. The other two follow from the projection-formula
  proofs of Theorem~\ref{thm:crepant-functoriality} and
  Proposition~\ref{prop:normalization-additivity}.
\end{proof}

\begin{definition}\label{def:psi-ksba}
Let $f\colon(X,D=\sum_{i=1}^na_iD_i)\to S$ be a KSBA family of
pure relative dimension $d$, and suppose that some of the divisors
$D_i$ are relative $\bQ$-Cartier. Let us denote
$c_1\bigl(\cO_X(D_i)\bigr)$ simply by $D_i$ whenever $D_i$ is
$\bQ$-Cartier. For $m\ge0$ and
$\boldm=(m_1,\dotsc,m_n)\in\bZ_{\ge0}^n$, with $m_i=0$ whenever
$D_i$ is not $\bQ$-Cartier, set $|\boldm|=\sum_i m_i$. If
$r=m+|\boldm|-d\ge0$, define
\begin{displaymath}
  \mu(f;m,\boldm)
    :=f_*\Bigl((K_f+D)^m\prod_{i=1}^nD_i^{m_i}\Bigr)
    \in A^r(S)_\bQ.
\end{displaymath}
We call these the \emph{mixed kappa classes} of the family and omit the
argument $f$ when the family is clear.
\end{definition}

For $\ocM_{g,n}$ one has $(K_f+D)D_i=0$ and
$D_iD_j=0$ for $i\ne j$. Hence the only mixed classes that
can be nonzero are
\begin{displaymath}
  \mu(r+1,\bzero)=\kappa_r,
  \qquad
  \mu\bigl(0,(r+1)\be_i\bigr)=(-1)^r\psi_i^r.
\end{displaymath}
For the moduli spaces $\ocM_{g,\ba}$ of weighted stable curves there
may exist additional nonzero mixed classes coming from collisions of
sections, cf. Example~\ref{ex:fedorchuk}. 

\begin{remark}
  When the individual $D_i$ are not $\bQ$-Cartier but some linear
  combinations $\sum c_i D_i$ are, one can define even more general mixed
  classes.
\end{remark}

\begin{definition}
  Under the additional stability hypothesis of
  Definition~\ref{def:psi-general}, we can define any mixed class on
  the normalized simplified moduli space $\cN_\ba$ by using the
  canonical identification $\cN_\ba = \cN_{\ba-\beps}$ for $0\le \beps
  \ll 1$, choosing $\ba-\beps\in\cP$ in an $n$-dimensional chamber.
  On its chamber model each $\cD_i$ is relative
  $\bQ$-Cartier, since we can vary each coefficient independently.
  Provided such a chamber exists, this definition depends on a finite
  choice of a chamber lying below $\ba$.
\end{definition}

\begin{proposition}[Restriction to a marked divisor]
  \label{prop:mixed-restriction}
  Suppose that $D_i$ is relative $\bQ$-Cartier, let
  $\iota_i\colon D_i\to X$ be the inclusion, and put
  $f_i=f\circ\iota_i$. Assume that $f_i$ is flat with pure relative
  dimension $d-1$. This hypothesis holds if $S$ is reduced, in
  particular for a family over $\cN_\ba$. Then
  \begin{displaymath}
    \mu(f;m,\boldm+\be_i)
      =(f_i)_*\Bigl(
        ((K_f+D)|_{D_i})^m
        \prod_{j=1}^n(D_j|_{D_i})^{m_j}
      \Bigr)
  \end{displaymath}
  whenever the class on the left is defined, with factors for
  $m_j=0$ omitted. In particular,
  \begin{displaymath}
    \mu(f;m,\be_i)
      =(f_i)_*\bigl(((K_f+D)|_{D_i})^m\bigr).
  \end{displaymath}
\end{proposition}

\begin{proof}
  When $S$ is reduced, flatness and purity follow from
  \cite[Corollary~4.33, case~(2.79.1), and
  Definition--Lemma~4.19]{kollar2023families-of-varieties}, after
  pullback to a scheme atlas. To prove the operational identity, base
  change to an integral test scheme $T\to S$. The rational
  divisor-cycle identity gives
  \begin{displaymath}
    (\iota_{i,T})_*[D_{i,T}]
      =D_{i,T}\cap[X_T].
  \end{displaymath}
  Intersecting with the remaining classes and applying the projection
  formula proves the required equality on $[T]$, and hence
  operationally.
\end{proof}

\begin{remark}[Flat complete intersections]
  \label{rem:flat-complete-intersection}
  Suppose that distinct relative Cartier divisors
  $D_{i_1},\dotsc,D_{i_q}$ meet in a flat relative complete
  intersection
  \begin{displaymath}
    D_I=D_{i_1}\cap\dotsb\cap D_{i_q}\longrightarrow S
  \end{displaymath}
  of pure relative dimension $d-q$. The Cartier complete-intersection
  formula and the projection formula, applied after base change, give
  \begin{displaymath}
    \mu\bigl(f;m,\be_{i_1}+\dotsb+\be_{i_q}\bigr)
      =(f_I)_*\bigl(((K_f+D)|_{D_I})^m\bigr),
  \end{displaymath}
  where $f_I\colon D_I\to S$.
\end{remark}

Fix an $n$-dimensional chamber $\cP_j^0\subset\bR^n$ in the decomposition of
Section~\ref{sec:wall-crossing}. On its chamber family
$\ff_j\colon(\cX_j;\cD_{j,1},\dotsc,\cD_{j,n})\to\cM_j$, the divisors
$\cD_{j,i}$ are relative $\bQ$-Cartier: varying only $b_i$ within the
chamber expresses a nonzero rational multiple of $\cD_{j,i}$ as the
difference of two relative log canonical $\bQ$-Cartier divisors. Let us denote
$c_1\bigl(\cO_{\cX_j}(\cD_{j,i})\bigr)$ simply by $\cD_{j,i}$. For
$m,\boldm$ with $r=m+|\boldm|-d\ge0$, and
$\bb\in\cP_j\cap\bQ^n$, define
\begin{displaymath}
  \mu_j(m,\boldm;\bb)
    :=(\ff_j)_*\Bigl(
      (K_{\ff_j}+\bb\cD_j)^m\prod_{i=1}^n\cD_{j,i}^{m_i}
    \Bigr)\in A^r(\cM_j)_\bQ.
\end{displaymath}

\begin{lemma}\label{lem:derivatives}
Let $r=m+|\boldm|-d\ge0$. Then
\begin{displaymath}
  \mu_j(m,\boldm;\bb)
    =\frac{m!}{(r+d)!}
      \frac{\partial^{|\boldm|}}
        {\partial b_1^{m_1}\dotsm\partial b_n^{m_n}}
      \kappa_r(\ff_j,\bb).
\end{displaymath}
\end{lemma}

\begin{proof}
  Put $L(\bb)=K_{\ff_j}+\bb\cD_j$. On the fixed chamber model,
  $\partial L/\partial b_i=\cD_{j,i}$, so
  \begin{displaymath}
    \frac{\partial^{|\boldm|}}
      {\partial b_1^{m_1}\dotsm\partial b_n^{m_n}}
      L(\bb)^{r+d}
      =\frac{(r+d)!}{m!}\,L(\bb)^m
        \prod_{i=1}^n\cD_{j,i}^{m_i}.
  \end{displaymath}
  Applying $(\ff_j)_*$, which commutes with differentiation, proves
  the formula.
\end{proof}

\begin{corollary}[Differential recursion]
  \label{cor:mixed-differential-recursion}
  If $m\ge1$, then
  \begin{displaymath}
    \frac{\partial}{\partial b_i}\mu_j(m,\boldm;\bb)
      =m\,\mu_j(m-1,\boldm+\be_i;\bb).
  \end{displaymath}
  More generally, if
  $\bm{\alpha}=(\alpha_1,\dotsc,\alpha_n)\in\bZ_{\ge0}^n$ and
  $|\bm{\alpha}|\le m$, then
  \begin{displaymath}
    \frac{\partial^{|\bm{\alpha}|}}
      {\partial b_1^{\alpha_1}\dotsm\partial b_n^{\alpha_n}}
    \mu_j(m,\boldm;\bb)
      =\frac{m!}{(m-|\bm{\alpha}|)!}
        \mu_j\bigl(
          m-|\bm{\alpha}|,\boldm+\bm{\alpha};\bb
        \bigr).
  \end{displaymath}
  If $\bb+\bh\in\cP_j$, write
  $\bm{\alpha}!=\prod_i\alpha_i!$ and
  $\bh^{\bm{\alpha}}=\prod_i h_i^{\alpha_i}$. Then
  \begin{displaymath}
    \mu_j(m,\boldm;\bb+\bh)
      =\sum_{|\bm{\alpha}|\le m}
        \frac{m!}{(m-|\bm{\alpha}|)!\bm{\alpha}!}
        \bh^{\bm{\alpha}}
        \mu_j\bigl(
          m-|\bm{\alpha}|,\boldm+\bm{\alpha};\bb
        \bigr).
  \end{displaymath}
\end{corollary}

\begin{proof}
  Differentiate the definition of $\mu_j$. The last formula is the
  binomial expansion of $(K_{\ff_j}+(\bb+\bh)\cD_j)^m$.
\end{proof}

For weighted stable curves, if $\psi_i^{(j)}$ denotes the psi class on
$\cM_j$, this gives
\begin{displaymath}
  \psi_i^{(j)}=-\mu_j(0,2\be_i;\bb)
    =-\frac12\frac{\partial^2}{\partial b_i^2}
      \kappa_1(\ff_j,\bb).
\end{displaymath}
For $i\ne k$, the mixed derivative
$\frac12\partial^2\kappa_1(\ff_j,\bb)/\partial b_i\partial b_k$ is the
collision class $\Delta_{ik}$ on the same chamber.

\begin{remark}[Tangential compatibility at a wall]
  \label{rem:tangential-wall}
  Suppose that $\cP_k$ is a face of $\cP_j$, let
  $\bb\in\cP_k\cap\bQ^n$, and let
  $\bh_1,\dotsc,\bh_q\in
    \operatorname{Span}_\bR(\cP_k-\cP_k)\cap\bQ^n$.
  Write
  \begin{displaymath}
    \bh_a\cD_j:=\sum_i h_{a,i}\cD_{j,i},
  \end{displaymath}
  and use the same notation for its first Chern class. Define
  $\bh_a\cD_k$ similarly. If $q\le r+d$, then
  \begin{displaymath}
    (\ff_j)_*\Bigl(
      (K_{\ff_j}+\bb\cD_j)^{r+d-q}
      \prod_{a=1}^q(\bh_a\cD_j)
    \Bigr)
      =\alpha_{jk}^*(\ff_k)_*\Bigl(
        (K_{\ff_k}+\bb\cD_k)^{r+d-q}
        \prod_{a=1}^q(\bh_a\cD_k)
      \Bigr).
  \end{displaymath}
  Indeed, this follows by differentiating
  Theorem~\ref{thm:operational-wall-crossing} in directions tangent to
  $\cP_k$. Thus mixed derivatives tangent to a common face are
  independent of the adjacent $n$-dimensional chamber; derivatives
  involving transverse directions may depend on it.
\end{remark}

The numerical nonnegativity of kappa classes extends to the mixed
classes $\mu(m,\be_i)$ for $m\ge d-1$.
Fix an $n$-dimensional chamber $\cP_j^0\subset\bR^n$ and a morphism
$S\to\cM_j$. Write
\begin{math}
  f\colon(X;D_1,\dotsc,D_n)\to S
\end{math}
for the pulled-back chamber family.
For $\bb\in\cP_j\cap\bQ^n$, write
$\bb D=\sum_{i=1}^n b_iD_i$, so that $K_f+\bb D$ is the pullback of
$K_{\ff_j}+\bb\cD_j$. The class $D_i$ is the pullback of
$\cD_{j,i}$.

\begin{proposition}
For every $\bb\in\cP_j\cap\bQ^n$, every $r\ge0$, every $i$, and every
proper integral subscheme $Z\subset S$ of dimension $r$, one has
\begin{displaymath}
  \deg\Bigl(
    f_*\bigl((K_f+\bb D)^{r+d-1}D_i\bigr)\cap[Z]
  \Bigr)\ge0.
\end{displaymath}
\end{proposition}

\begin{proof}
By base change and the projection formula, we may replace $S$ by a
smooth projective variety mapping properly and birationally to $Z$,
as in the proof of Proposition~\ref{prop:kappa-nonnegative}.
Put $L=K_f+\bb D$. For interior weights $L$ is nef by the same
positivity result used there; at boundary weights it is nef as a
limit of these classes. Since $D_i$ is effective, the desired degree
is $\deg_X(L^{r+d-1}\cap[D_i])\ge0$.
\end{proof}

\section{Riemann--Roch and formulas for
  kappa  and mixed classes}
\label{sec:GRR}

In this section we use the Grothendieck--Riemann--Roch theorem and its
variation due to Baum--Fulton--MacPherson to derive a formula for a
Gysin pushforward and, as a corollary, a formula for $\kappa_r$. One
difficulty is that BFM Riemann-Roch is an identity in Chow groups,
but we need an identity in operational Chow cohomology. That's why we
can not just apply it immediately and have to do some work.

In the scheme statements of this section, the base $Y$ is assumed to
be separated. Operational Chow groups retain their full test category,
including nonseparated schemes.


\subsection{A formula for Gysin pushforward under a proper flat
  morphism}
\label{sec:gysin-formula}

There are three basic Grothendieck groups associated with a scheme $X$:

\begin{enumerate}
\item $K_0(\Coh X) = K_0(D^b(\Coh X))$, the Grothendieck group of coherent
  sheaves on $X$. Here, $D^b(\Coh X)$ is the bounded derived category of
  coherent sheaves on~$X$.
\item $K_0(\Perf X) = K_0(D^\perfect(X))$. Here, $D^\perfect(X)$ is the full
  triangulated subcategory of $D^b(\Coh X)$ consisting of perfect complexes,
  which are locally quasi-isomorphic to bounded complexes of finite-rank
  vector bundles.
\item $K_0(\Vect X)$, the Grothendieck group of locally free sheaves
  on $X$.
\end{enumerate}
There are natural homomorphisms
\begin{displaymath}
  K_0(\Vect X) \xrightarrow{\alpha_X}
  K_0(\Perf X) \xrightarrow{\beta_X} K_0(\Coh X).
\end{displaymath}
The map $\alpha_X$ is an isomorphism when $X$ has the resolution
property; for example, this holds when $X$ is quasi-projective over a
field $k$. It need not be an isomorphism for a general scheme of
finite type over $k$. The map $\beta_X$ is an isomorphism if $X$ is
smooth, but rarely otherwise.

\begin{remark}
  Notation for these groups varies widely in the literature. The book
  \cite{fulton1998intersection-theory}, our main reference for
  intersection theory, uses the following notations:
  $K_0 X$ for $K_0(\Coh X)$ and $K^0 X$ for $K_0(\Vect X)$.
\end{remark}

We recall that for a proper morphism $f$ there is a pushforward
homomorphism
\begin{displaymath}
  f_!\colon K_0(\Coh X) \longrightarrow K_0(\Coh Y), \qquad
  [\cF] \longmapsto \sum_{k\ge0} (-1)^k [R^k f_* \cF].
\end{displaymath}

If $f$ is proper and flat, derived pushforward induces a homomorphism
\begin{displaymath}
  \dR f_*\colon K_0(\Perf X)\longrightarrow K_0(\Perf Y).
\end{displaymath}
This pushforward is compatible with coherent pushforward and commutes
with arbitrary base change. These assertions follow from
\cite[\href{https://stacks.math.columbia.edu/tag/0B91}
{Lemma~36.30.4(1)}]{stacks-project}.

Vector bundles (and perfect complexes as well) have operational Chern
classes and Chern characters, and any $P\in K_0(\Perf X)$ has Chern
classes and Chern characters in operational Chow cohomology; see
\cite[Chapter~3]{fulton1998intersection-theory} and
\cite[\href{https://stacks.math.columbia.edu/tag/0ESY}{Section~42.46}]{stacks-project}.

\begin{theorem}\label{thm:gysin-formula}
  Let $r\geq0$, and let $f\colon X\to Y$ be a proper flat morphism
  of pure relative dimension~$d$. Let $P\in K_0(\Vect X)$ and suppose that
  $\ch_j(P)=0$ for $j<r+d$. Write
  \begin{displaymath}
    \ch(P)=\ch(P)_{r+d}+\text{(terms of degree $>r+d$)}.
  \end{displaymath}
  Then the following identities hold in $A^*(Y)_\bQ$:
  \begin{enumerate}
  \item  $\ch_j(\dR f_* P) = 0$ for $j<r$, and
    $\ch_r(\dR f_* P)=f_*\ch(P)_{r+d}$.
  \item If $r\ge1$, then 
    \begin{displaymath}
      f_* \ch(P)_{r+d}
      = \frac{(-1)^{r-1}}{(r-1)!} c_r (\dR f_* P).
    \end{displaymath}
  \end{enumerate}
\end{theorem}
\begin{proof}
  We prove part (1). Part (2) then follows from
  Lemma~\ref{lem:newton}.

  Chow's lemma, resolution of singularities in characteristic zero,
  and induction on dimension give a Chow envelope
  $\rho\colon T\to Y$, where $T$ is a finite disjoint union of smooth
  integral quasi-projective schemes, possibly of different dimensions.
  By \cite[Example~17.3.2]{fulton1998intersection-theory}, pullback
  along $\rho$ is injective on operational Chow cohomology. By
  \cite[Corollary~17.4]{fulton1998intersection-theory}, an operational
  class on each smooth component is determined by its action on the
  fundamental class. Both statements use the full operational test
  category. Thus it suffices to evaluate on each such component,
  which we again denote by~$T$.

  Write $m=\dim T$, $X_T=X\times_Y T$, and
  $f_T\colon X_T\to T$. Proper flat base change identifies
  $\dR(f_T)_*P_T$ with the pullback of $\dR f_*P$, where
  $P_T=P|_{X_T}$; see
  \cite[\href{https://stacks.math.columbia.edu/tag/0B91}{Lemma~36.30.4(1)}]{stacks-project}.
  Chern classes commute with this pullback by
  \cite[\href{https://stacks.math.columbia.edu/tag/0FAC}{Lemma~42.46.6}]{stacks-project}.
  Since $T$ is quasi-projective, the perfect class $\dR(f_T)_*P_T$
  lies in the image of $K_0(\Vect T)\to K_0(\Perf T)$.

  Let $\tau_W\colon K_0(\Coh W)\to A_*(W)_\bQ$ denote the singular
  Riemann--Roch transformation, and put
  $\Td(W)=\tau_W([\cO_W])$. The Baum--Fulton--MacPherson theorem,
  in the form of \cite[Cor.~18.3.1]{fulton1998intersection-theory}
  (after \cite{baum1975riemann-roch, fulton1983riemann-roch}), gives
  \begin{equation}\label{eq:RR}
    \ch\bigl(\dR(f_T)_*P_T\bigr)\cap[T]
      =\td(T_T)^{-1}\cap
        (f_T)_*\bigl(\ch(P_T)\cap\Td(X_T)\bigr).
  \end{equation}
  Here we used $\Td(T)=\td(T_T)\cap[T]$. Since $X_T$ has pure
  dimension $m+d$, \cite[18.3(5)]{fulton1998intersection-theory}
  gives
  \begin{displaymath}
    \Td(X_T)=[X_T]+\text{(terms of dimension $<m+d$)}.
  \end{displaymath}
  The right side of \eqref{eq:RR} has no component of dimension
  greater than $m-r$, and its component of dimension $m-r$ is
  \begin{displaymath}
    (f_T)_*\bigl(\ch(P_T)_{r+d}\cap[X_T]\bigr),
  \end{displaymath}
  because $\td(T_T)^{-1}$ has constant term~$1$. Comparing
  dimensions proves both assertions of part~(1) on $[T]$. Poincar\'e
  duality gives the operational identities on each component of the
  envelope, and injectivity of $\rho^*$ gives them on~$Y$.
\end{proof}

We recall that vector bundles and, more generally, perfect complexes
have Chern classes and Chern characters, and that there are universal
identities between them; cf.
\cite[\href{https://stacks.math.columbia.edu/tag/0GUD}{Lemma~42.46.2}]{stacks-project},
\cite[\href{https://stacks.math.columbia.edu/tag/0FAD}{Remark~42.46.8}]{stacks-project}. For
example, $\ch_1 = c_1$, $\ch_2 = \frac12(c_1^2 -2c_2)$,
$\ch_3 = \frac16(c_1^3 - 3c_1c_2 +3c_3)$, etc. In terms of Chern roots
$x_s$, the Chern class $c_k$ is the $k$-th elementary symmetric
polynomial $e_k(x_s)$, and the $k$-th component of the Chern character,
$\ch_k$, is the scaled power sum
\begin{displaymath}
  \ch_k = \frac1{k!} \sum_{s} x_s^k = \frac1{k!} p_k(x_s).
\end{displaymath}

\begin{lemma}\label{lem:newton}
  Let $r\ge1$ and let $V\in K_0(\Perf X)$. If
  $\ch_1(V)=\dotsb=\ch_{r-1}(V)=0$, then
  \begin{displaymath}
    \ch_r(V)=(-1)^{r-1}\frac{c_r(V)}{(r-1)!}.
  \end{displaymath}
\end{lemma}
\begin{proof}
  For any $r\ge1$, the Newton identity says
  \begin{displaymath}
    r e_r = \sum_{j=1}^r (-1)^{j-1} p_j e_{r-j}.
  \end{displaymath}
  If $p_1 = \dotsb = p_{r-1} = 0$, then $r e_r = (-1)^{r-1} p_r$,
  which translates to $r c_r = (-1)^{r-1} r!\ch_r$.
\end{proof}
  In particular, the lemma applies to the perfect class
  $V=\dR f_*P$ occurring in Theorem~\ref{thm:gysin-formula}.

We first record a multilinear consequence of that theorem.

\begin{proposition}[The multilinear formula]\label{prop:AL-formula}
  Let $r\ge0$, let $f\colon X\to Y$ be proper and flat of pure
  relative dimension $d$, and put $s=r+d$. For line bundles
  $L_1,\dotsc,L_s,A$ on $X$, set
  \begin{displaymath}
    P_A=[A]\prod_{a=1}^s([L_a]-1)\in K_0(\Vect X).
  \end{displaymath}
  Then $\ch_q(\dR f_*P_A)=0$ for $q<r$ and
  \begin{displaymath}
    f_*\Bigl(\prod_{a=1}^s c_1(L_a)\Bigr)
      =\ch_r(\dR f_*P_A)
      =\sum_{I\subset\{1,\dotsc,s\}}(-1)^{s-|I|}
        \ch_r\Bigl(\dR f_*
          \bigl(A\otimes\mathop{\bigotimes}_{a\in I}L_a\bigr)\Bigr).
  \end{displaymath}
  If $r\ge1$, then
  \begin{displaymath}
    f_*\Bigl(\prod_{a=1}^s c_1(L_a)\Bigr)
      =\frac{(-1)^{r-1}}{(r-1)!}\,c_r(\dR f_*P_A).
  \end{displaymath}
\end{proposition}

\begin{proof}
  Write $\ell_a=c_1(L_a)$ and $\alpha=c_1(A)$. Then
  \begin{displaymath}
    \ch(P_A)=e^\alpha\prod_{a=1}^s(e^{\ell_a}-1)
      =\prod_{a=1}^s\ell_a+
        \text{(terms of degree greater than $s$)}.
  \end{displaymath}
  Theorem~\ref{thm:gysin-formula} gives the vanishing and the
  Chern-character and Chern-class identities. Expanding $P_A$ gives
  the alternating sum.
\end{proof}

For projective $f$, the case $r=1$ and $A=\cO_X$ recovers the
first-Chern-class formula for the Deligne pairing of line bundles
$L_1,\dotsc,L_s$; see
\cite{biswas2011deligne-pairing} and
\cite[Section~6]{eriksson2024deligne-riemann-roch}.

\subsection{Chern character formulas for kappa classes}
\label{sec:kappa-CM}

Specializing to powers of a single line bundle gives formulas for
$\kappa_r$ as a Chern character, as a Chern class, and as a coefficient
of a polynomial in the tensor power.

\begin{proposition}[The diagonal formula]\label{prop:diagonal-AL-formula}
  Let $r\ge0$, let $f\colon X\to Y$ be proper and flat of pure
  relative dimension $d$, and put $s=r+d$. For line bundles $L,A$ on
  $X$, write $\ell=c_1(L)$ and set
  \begin{displaymath}
    P_A=[A]([L]-1)^s\in K_0(\Vect X).
  \end{displaymath}
  Then $\ch_q(\dR f_*P_A)=0$ for $q<r$, and
  \begin{displaymath}
    f_*(\ell^s)
      =\ch_r(\dR f_*P_A)
      =\sum_{j=0}^s(-1)^{s-j}\binom sj
        \ch_r\bigl(\dR f_*(A\otimes L^{\otimes j})\bigr).
  \end{displaymath}
  If $r\ge1$, then
  \begin{displaymath}
    f_*(\ell^s)
      =\frac{(-1)^{r-1}}{(r-1)!}\,c_r(\dR f_*P_A).
  \end{displaymath}
\end{proposition}

\begin{proof}
  Apply Proposition~\ref{prop:AL-formula} with
  $L_1=\dotsb=L_s=L$, grouping subsets by their cardinality.
\end{proof}

Keep $f$ as above, let $L$ be an $f$-ample line bundle, and put
$\ell=c_1(L)$. For an integer $u$, set
\begin{displaymath}
  \bV_u:=\dR f_*L^{\otimes u}\in K_0(\Perf Y).
\end{displaymath}
We will show that $\ch_r(\bV_u)$ depends polynomially on $u$ and
identify its coefficient of $u^{r+d}$.
For $u\gg0$, relative Serre vanishing and cohomology and base change
show that
\begin{displaymath}
  V_u:=f_*L^{\otimes u}
\end{displaymath}
is a vector bundle, compatible with arbitrary base change, and
$\bV_u=[V_u]$.

Below, we use the notation $\Delta G(u)=G(u+1)-G(u)$ for the upper
discrete derivative.

\begin{corollary}[Coefficient of $u^{r+d}$]\label{cor:kappa-leading-coefficient}
  Put $s=r+d$. The function $u\mapsto\ch_r(\bV_u)$ is a polynomial
  sequence of degree at most $s$, and
  \begin{displaymath}
    \ch_r(\bV_u)
      =\frac{1}{s!}\,f_*(\ell^s)u^s+P_{s-1}(u),
    \qquad P_{s-1}(u)\in A^r(Y)_\bQ[u].
  \end{displaymath}
  Here, when $s=0$, we use the convention $P_{-1}=0$. Thus, for
  $u\gg0$, the class $f_*(\ell^s)$ is $s!$ times the coefficient
  of $u^s$ in $\ch_r(V_u)$.
\end{corollary}

\begin{proof}
  Proposition~\ref{prop:diagonal-AL-formula}, applied with
  $A=L^{\otimes u}$, gives
  \begin{displaymath}
    \Delta^s\ch_r(\bV_u)=f_*(\ell^s).
  \end{displaymath}
  Hence $\Delta^{s+1}\ch_r(\bV_u)=0$. Writing the resulting polynomial
  sequence in the binomial basis gives the formula.
\end{proof}

Let us apply these formulas to the universal family
$\ff\colon(\cX,\cD)\to\cM$. Since $\cM$ is quasi-compact and $\cL_N$
is relatively ample, relative Serre vanishing can be checked on a
finite-type scheme atlas. Thus there is an integer $m_0$ such that,
for $u\ge m_0$,
\begin{displaymath}
  \cV_u:=\ff_*\cL_N^{\otimes u}
\end{displaymath}
is a vector bundle on $\cM$, compatible with arbitrary base change.

The universal identities below can be checked after a finite flat
surjective morphism $h\colon Q\to\cM$ from a quasi-projective scheme.
Such a cover exists by \cite{kresch2004on-coverings}, since $\cM$ is
a separated DM global quotient with projective coarse moduli space.
Pullback along $h$ is injective on rational operational Chow
cohomology: for any integral scheme test $S\to\cM$, including a
nonseparated one, the base change $h_S$ satisfies
$(h_S)_*h_S^*=e\,\id$ on $A_*(S)_\bQ$, where $e>0$ is its degree.
Compatibility with flat pullback and proper pushforward gives the
claimed injectivity. Thus the scheme formulas over the separated
base $Q$ descend to full operational identities on~$\cM$, and hence
hold after arbitrary base change.

\begin{corollary}[Universal vector-bundle and Chern-class formulas]
  \label{cor:universal-vector-bundle}
  For $r\ge0$ and $u\ge m_0$, put
  \begin{displaymath}
    \cE_{r,u}
      :=\sum_{j=0}^{r+d}(-1)^{r+d-j}\binom{r+d}{j}
        [\cV_{u+j}]\in K_0(\Vect\cM).
  \end{displaymath}
  Then
  \begin{displaymath}
    \ch_p(\cE_{r,u})=0\quad\text{for }p<r,
    \qquad
    \ch_r(\cE_{r,u})=N^{r+d}\kappa_r,
  \end{displaymath}
  and
  \begin{displaymath}
    \kappa_r
      =\frac{1}{N^{r+d}}
        \sum_{j=0}^{r+d}(-1)^{r+d-j}\binom{r+d}{j}
          \ch_r(\cV_{u+j}).
  \end{displaymath}
  If $r\ge1$, then
  \begin{displaymath}
    c_1(\cE_{r,u})=\dotsb=c_{r-1}(\cE_{r,u})=0
  \end{displaymath}
  and
  \begin{displaymath}
    \kappa_r
      =\frac{(-1)^{r-1}}{N^{r+d}(r-1)!}\,c_r(\cE_{r,u}).
  \end{displaymath}
  Finally, for $u\gg0$,
  \begin{displaymath}
    \ch_r(\cV_u)
      =\frac{N^{r+d}}{(r+d)!}\,\kappa_r u^{r+d}
        +P_{r+d-1}(u),
    \qquad P_{r+d-1}(u)\in A^r(\cM)_\bQ[u].
  \end{displaymath}
  Thus $\kappa_r$ is $(r+d)!/N^{r+d}$ times the coefficient of
  $u^{r+d}$ in $\ch_r(\cV_u)$.
\end{corollary}

\begin{proof}
  By the finite flat descent just explained, we may work over~$Q$.
  Expanding $[\cL_N]^u([\cL_N]-1)^{r+d}$ and applying relative
  Serre vanishing identifies its derived pushforward with
  $\cE_{r,u}$. Proposition~\ref{prop:diagonal-AL-formula} gives the
  formulas for its Chern character, and
  Lemma~\ref{lem:newton} gives the assertions about Chern classes.
  The last assertion follows from
  Corollary~\ref{cor:kappa-leading-coefficient} and
  $\ff_*\ell_N^{r+d}=N^{r+d}\kappa_r$.
\end{proof}

\subsection{Chern character formulas for mixed classes}
\label{sec:mixed-CM}

We now apply Proposition~\ref{prop:AL-formula} to mixed classes.

For the vector-bundle version, let $f\colon X\to Y$ be projective and
flat of pure relative dimension $d$, let $L$ be $f$-ample, and let
$B_1,\dotsc,B_n$ be line bundles. Put $\ell=c_1(L)$ and
$\beta_i=c_1(B_i)$. Fix $m\ge0$ and
$\boldm=(m_1,\dotsc,m_n)\in\bZ_{\ge0}^n$, with
$q=|\boldm|$ and $r=m+q-d\ge0$. For $\bm{\nu}\le\boldm$, write
\begin{displaymath}
  |\bm{\nu}|=\sum_i\nu_i,\qquad
  \binom{\boldm}{\bm{\nu}}=\prod_i\binom{m_i}{\nu_i},\qquad
  B^{\bm{\nu}}=\bigotimes_i B_i^{\otimes\nu_i}.
\end{displaymath}
For all sufficiently large $u$, the finitely many sheaves
$V_{u,\bm{\nu}}:=f_*(L^{\otimes u}\otimes B^{\bm{\nu}})$ are vector
bundles, compatible with arbitrary base change. Choose a single
$u_0$ that works for every $\bm{\nu}\le\boldm$ and every $u\ge u_0$.

\begin{proposition}[Mixed vector-bundle formulas]
  \label{prop:mixed-vector-bundle}
  For $u\ge u_0$, set
  \begin{displaymath}
    E_{m,\boldm;u}
      :=\sum_{a=0}^m\sum_{\bm{\nu}\le\boldm}
        (-1)^{r+d-a-|\bm{\nu}|}\binom ma\binom{\boldm}{\bm{\nu}}
        [V_{u+a,\bm{\nu}}]\in K_0(\Vect Y).
  \end{displaymath}
  Then $\ch_p(E_{m,\boldm;u})=0$ for $p<r$, and
  \begin{displaymath}
    f_*\Bigl(\ell^m\prod_i\beta_i^{m_i}\Bigr)
      =\ch_r(E_{m,\boldm;u}).
  \end{displaymath}
  If $r\ge1$, the Chern classes of $E_{m,\boldm;u}$ of degrees
  $1,\dotsc,r-1$ vanish, and
  \begin{displaymath}
    f_*\Bigl(\ell^m\prod_i\beta_i^{m_i}\Bigr)
      =\frac{(-1)^{r-1}}{(r-1)!}\,c_r(E_{m,\boldm;u}).
  \end{displaymath}
\end{proposition}

\begin{proof}
  Expanding $[L]^u([L]-1)^m\prod_i([B_i]-1)^{m_i}$ and applying
  relative Serre vanishing identifies its derived pushforward with
  $E_{m,\boldm;u}$. Apply Proposition~\ref{prop:AL-formula} to $m$
  copies of $L$ and $m_i$ copies of $B_i$, with $A=L^{\otimes u}$,
  and use Lemma~\ref{lem:newton} for the Chern classes.
\end{proof}

For the universal family $\ff\colon(\cX,\cD)\to\cM$, take a mixed
class $\mu(m,\boldm)$ with $\boldm\ne\bzero$ and $r=m+|\boldm|-d\ge0$.
Choose $N$ sufficiently divisible that $N\cD_i$ is Cartier whenever
$m_i>0$, and put $\cB_i=\cO_\cX(N\cD_i)$ for these indices and
$\cB_i=\cO_\cX$ otherwise. Write
$\cB^{\bm{\nu}}=\bigotimes_i\cB_i^{\otimes\nu_i}$. As before, for
one sufficiently large $u_0$ all
\begin{displaymath}
  \cV_{u,\bm{\nu}}:=\ff_*(\cL_N^{\otimes u}\otimes\cB^{\bm{\nu}}),
  \qquad u\ge u_0,\quad\bm{\nu}\le\boldm,
\end{displaymath}
are vector bundles compatible with arbitrary base change.

\begin{corollary}[Universal mixed formulas]
  \label{cor:universal-mixed-vector-bundle}
  For $u\ge u_0$, define
  \begin{displaymath}
    \cE_{m,\boldm;u}
      :=\sum_{a=0}^m\sum_{\bm{\nu}\le\boldm}
        (-1)^{r+d-a-|\bm{\nu}|}\binom ma\binom{\boldm}{\bm{\nu}}
        [\cV_{u+a,\bm{\nu}}].
  \end{displaymath}
  Then $\ch_p(\cE_{m,\boldm;u})=0$ for $p<r$, and
  \begin{displaymath}
    \mu(m,\boldm)=\frac1{N^{r+d}}\ch_r(\cE_{m,\boldm;u}).
  \end{displaymath}
  If $r\ge1$, the Chern classes of $\cE_{m,\boldm;u}$ of degrees
  $1,\dotsc,r-1$ vanish, and
  \begin{displaymath}
    \mu(m,\boldm)
      =\frac{(-1)^{r-1}}{N^{r+d}(r-1)!}\,c_r(\cE_{m,\boldm;u}).
  \end{displaymath}
\end{corollary}

\begin{proof}
  After pullback to the finite flat cover $Q\to\cM$, apply
  Proposition~\ref{prop:mixed-vector-bundle} with $L=\cL_N$
  and $B_i=\cB_i$, and descend as above, using
  $\ff_*(\ell_N^m\prod_i c_1(\cB_i)^{m_i})=N^{r+d}\mu(m,\boldm)$.
\end{proof}

For insertions coming from divisors, the coefficient formula
also has a direct restriction interpretation.

\begin{corollary}[Coefficient of $u^{r+d-1}$ on a flat divisor]
  \label{cor:flat-divisor-leading-coefficient}
  Let $f\colon X\to Y$ and $L$ be as above. Suppose $d\ge1$ and
  $D\subset X$ is a relative $\bQ$-Cartier divisor such that
  $f_D\colon D\to Y$ is flat of pure relative dimension $d-1$.
  For $r\ge0$, put $m=r+d-1$ and $U_u=(f_D)_*(L^{\otimes u}|_D)$.
  For $u\gg0$, this is a vector bundle compatible with base change,
  and
  \begin{displaymath}
    \ch_r(U_u)=\frac1{m!}\,f_*(\ell^mD)u^m+P_{m-1}(u),
    \qquad P_{m-1}(u)\in A^r(Y)_\bQ[u].
  \end{displaymath}
\end{corollary}

\begin{proof}
  Apply Corollary~\ref{cor:kappa-leading-coefficient} to $f_D$ and
  $L|_D$, and use the projection formula.
\end{proof}

\begin{corollary}[Coefficient of $u^{r+d-1}$ for one marked divisor]
  \label{cor:universal-divisor-leading-coefficient}
  Suppose $d\ge1$ and $\cD_i$ is relative $\bQ$-Cartier, and assume
  that $\ff_i\colon\cD_i\to\cM$ is flat of pure relative dimension
  $d-1$. This flatness hypothesis holds over a reduced base, in
  particular for the corresponding family over $\cN_\ba$, by
  Proposition~\ref{prop:mixed-restriction}. For $r\ge0$, put
  $m=r+d-1$. For $u\gg0$,
  $\cU_{i,u}:=(\ff_i)_*(\cL_N^{\otimes u}|_{\cD_i})$ is a vector
  bundle compatible with base change, and
  \begin{displaymath}
    \ch_r(\cU_{i,u})
      =\frac{N^m}{m!}\,\mu(m,\be_i)u^m+P_{m-1}(u),
    \qquad P_{m-1}(u)\in A^r(\cM)_\bQ[u].
  \end{displaymath}
\end{corollary}

\begin{proof}
  After pullback to the same finite flat cover, apply the preceding
  corollary and descend, using
  $(\ff_i)_*((\ell_N|_{\cD_i})^m)=N^m\mu(m,\be_i)$.
\end{proof}

\begin{remark}[Coefficients on flat complete intersections]
  \label{rem:complete-intersection-leading}
  The same argument applies when distinct relative Cartier divisors
  $\cD_{i_1},\dotsc,\cD_{i_q}$ meet in a flat relative complete
  intersection $\ff_I\colon\cD_I\to\cM$ of relative dimension $d-q$.
  For $r\ge0$, put $m=r+d-q$ and
  $\cU_{I,u}:=(\ff_I)_*(\cL_N^{\otimes u}|_{\cD_I})$.
  For $u\gg0$ this is a vector bundle compatible with base change, and
  \begin{displaymath}
    \ch_r(\cU_{I,u})
      =\frac{N^m}{m!}\,
        \mu\bigl(m,\be_{i_1}+\dotsb+\be_{i_q}\bigr)u^m
        +P_{m-1}(u).
  \end{displaymath}
  Here $P_{m-1}(u)\in A^r(\cM)_\bQ[u]$ and $P_{-1}=0$, as in
  Corollary~\ref{cor:kappa-leading-coefficient}. The identification of
  the mixed class is Remark~\ref{rem:flat-complete-intersection}.
\end{remark}

\section{An example: Sylvester-type moduli spaces of Calabi-Yau
  varieties}
\label{sec:exampe-cy} 

As an example, we compute the tautological classes for the infinite
series of Calabi-Yau moduli spaces introduced in 
\cite{alexeev2025tower-complete}. Let $s_0=2$, $s_1=3$, $s_2=7$, $s_3 = 43$, \dots
be the Sylvester sequence, $d_n = \prod_{k=0}^n s_k = s_{n+1}-1$, and
$a_k = d_n/s_k$. Consider a hypersurface
\begin{displaymath}
  X  = \bigl(F(x_0, \dotsc, x_n, x_{n+1}) = 0\bigr)
  \subset \bP(a_0, \dotsc, a_n, 1) 
\end{displaymath}
of degree $d_n$ and dimension $n$. (Thus, the meaning of $d$ and $n$
in this section differs from that in the rest of the paper.) Over a
$\bZ[1/d_n]$-scheme, such a hypersurface is called Fermat-nondegenerate
if for $0\le k\le n$ the monomials $x_k^{s_k}$ appear in $F$ with
coefficients that are units on the base. We also
require that the deformation parameters in the normal form of
\cite[Lemma~7.3]{alexeev2025tower-complete} have no common zero on the
base, excluding the undeformed Brieskorn--Pham hypersurface on every
geometric fiber. Let
\begin{displaymath}
  D = (x_{n+1} = 0) \subset X
\end{displaymath}
be the infinite divisor. The construction in
\cite{alexeev2025tower-complete} gives a weighted projective stack
$\cE_n=\cP(\delta_I)$ over $\bZ[1/d_n]$, for explicit positive
integers $\delta_I$, carrying a universal family of such hypersurface
pairs. After base change to $\bC$, the pairs $(X,D)$ are KSBA-stable
and $\cE_n$ is identified with their KSBA moduli stack. For example,
$\cE_1=\ocM_{1,1}=\cP(4,6)$ over $\bZ[1/6]$.
In this section we work over $\bZ[1/d_n]$, using rational equivariant
Chow theory for these quotient stacks. The kappa, psi, and mixed
classes and all the identities below are defined over $\bZ[1/d_n]$.

\Needspace{4\baselineskip}
\begin{theorem}
  Let $\ff\colon (\cX, \cD) \to \cE_n$ be the universal family, and
  $H=c_1\bigl( \cO_{\cE_n}(1) \bigr)$.
  \begin{enumerate}
  \item The kappa classes are $\displaystyle\kappa_r =
    \frac1{d_n^{n-1}} H^r$. Thus, their additive generating function is
    \begin{displaymath}
      \sum_{r\ge0} \kappa_r t^r = \frac1{d_n^{n-1}} \frac1{1- Ht}.
    \end{displaymath}
  \item The generalized psi classes for the divisor $\cD$ are
    \begin{displaymath}
      \psi_{1,r} = (-1)^r \ff_* \cD^{n+r} 
      = \frac1{d_n^{n-1}} \binom{n+r-1}{r} H^r 
      = \frac{1}{d_n^{n-1}}\binom{-n}{r} (-H)^r.
    \end{displaymath} 
    Thus, their additive generating function is
    \begin{displaymath}
      \sum_{r\ge0} \psi_{1,r} t^r = \frac1{d_n^{n-1}} \frac1{(1- Ht)^n}.
    \end{displaymath}
\item More generally, the mixed classes are
  \begin{displaymath}
    \mu(a,b) = \ff_* \bigl( (K_\ff+\cD)^a \cD^b \bigr) = \frac1{d_n^{n-1}}
    \binom{a-n}{r} H^r, \qquad a+b = n+r.
  \end{displaymath}
  \item For $n\ge1$ and $r\ge0$, one has the polynomial identity in $\epsilon$
    \begin{displaymath}
      \ff_*\bigl((K_\ff+\epsilon\cD)^{n+r}\bigr)
        = \epsilon^n  \frac{H^r}{d_n^{n-1}}
          \sum_{j=0}^r\binom{n+r}{j}
          \epsilon^{r-j}(1-\epsilon)^j.
    \end{displaymath}
  \end{enumerate}
\end{theorem}
\begin{proof}
  By \cite[Lemma~3.7 and Definition~3.13]{alexeev2025tower-complete}
  one has $K_\ff=\ff^*H$.  Set $A = K_\ff + \cD$ and let
  $g = \ff|_\cD\colon\cD\to \cE_n$ be the restriction of $\ff$ to
  $\cD$.

  By the projection formula,
  \begin{eqnarray}
    \label{eq:cy-kappas}
    \kappa_r = \ff_* \bigl( (K_\ff+ \cD) A^{n+r-1} \bigr)
    &=& \ff_* \bigl(\ff^*H \cdot A^{n+r-1}) + \ff_* \bigl( \cD A^{n+r-1} \bigr) \\
    &=& H\cdot \kappa_{r-1} + g_* \bigl(A|_\cD\bigr)^{n+r-1}. \nonumber
  \end{eqnarray}
  By \cite[Lemmas 3.4 and 3.12]{alexeev2025tower-complete},
  one has $\cD=\bP^{n-1}\times\cE_n$ and $g=p_2$ is the second
  projection. Let $h = c_1\bigl( p_1^*\cO_{\bP^{n-1}}(1) \bigr)$. We
  claim that $A|_\cD = h/d_n$. First, by the construction in
  \cite[Section~3]{alexeev2025tower-complete}, the powers
  $u_k=x_k^{s_k}$, $0\le k\le n$, are sections of
  $\cO_\cD(d_n\cD)\otimes g^*\cO_{\cE_n}(d_n)$ and identify
  $\cD$ with the hyperplane $\sum_{k=0}^n u_k=0$ in
  $\bP^n\times\cE_n$. Thus,
  \begin{displaymath}
    \cO_\cD(d_n\cD)\otimes g^*\cO_{\cE_n}(d_n)
      \simeq p_1^*\cO_{\bP^{n-1}}(1).
  \end{displaymath}
  Taking first Chern classes gives
  $d_n(\cD|_\cD+g^*H)=h$, and hence $A|_\cD=h/d_n$.

  Alternatively, let $h_k$ be the restriction to $\cD$ of the
  coordinate hyperplane $(u_k=0)\subset\bP^n\times\cE_n$.
  Each $h_k$ has class $h$, and $K_g=-nh$. The adjunction formula gives
  \begin{displaymath}
    \bigl(K_\ff + \cD)|_{\cD}
    =K_g+\sum_{k=0}^n\left(1-\frac1{s_k}\right)h_k
    \sim \Bigl(1 - \sum_{k=0}^n \frac1{s_k} \Bigr)h
      =\frac1{d_n}h.
  \end{displaymath}

  It follows that $g_* (A|_\cD)^a = 1/d_n^{n-1}$ for $a=n-1$ and zero for
  $a\ne n-1$.  Together with~\eqref{eq:cy-kappas} this proves part (1).

  Since $\cD|_\cD=h/d_n-g^*H$, for $b\ge1$ and $a+b=n+r$, we have
  \begin{align*}
    \mu(a,b)
      &=g_*\left(\left(\frac{h}{d_n}\right)^a
        \left(\frac{h}{d_n}-g^*H\right)^{b-1}\right)\\
      &=\frac{(-1)^r}{d_n^{n-1}}\binom{b-1}{r}H^r
        =\frac1{d_n^{n-1}}\binom{a-n}{r}H^r.
  \end{align*}
  Here $g_*$ extracts the coefficient of $h^{n-1}$. This proves
  part (3), with $b=0$ supplied by part (1). Setting $a=0$ and
  multiplying by $(-1)^r$ gives the psi formula in part (2), and the
  binomial theorem gives its generating function.

  Finally, expand
  \begin{displaymath}
    K_\ff+\epsilon\cD=\epsilon A+(1-\epsilon)\ff^*H.
  \end{displaymath}
  Terms with more than $r$ factors of $\ff^*H$ push forward to zero.
  Applying part (1) and the projection formula to the remaining terms
  gives the sum in part (4).
\end{proof}

\begin{example}
  For $n=2$, $\cE_2$ is the Baily-Borel compactification of
  the moduli stack of $U\oplus E_8$-polarized K3
  surfaces.  One has $d_2=2\cdot3\cdot7=42$, and the formulas above
  give
  \begin{displaymath}
    \kappa_1=\frac{H}{42},\qquad \psi_{1,1}=\frac{H}{21},
    \qquad
    \psi_{1,2}=\frac{H^2}{14} \ne\frac{H^2}{441}=\psi_{1,1}^{\,2}.
  \end{displaymath}
  Thus the generalized psi class $\psi_{1,2}$ differs from the square
  of $\psi_{1,1}$.
\end{example}

\bibliographystyle{amsalpha}
\newcommand{\etalchar}[1]{$^{#1}$}
\providecommand{\noopsort}[2]{#1}\def\cprime{$'$}
\providecommand{\bysame}{\leavevmode\hbox to3em{\hrulefill}\thinspace}
\providecommand{\MR}{\relax\ifhmode\unskip\space\fi MR }
\providecommand{\MRhref}[2]{%
  \href{http://www.ams.org/mathscinet-getitem?mr=#1}{#2}
}
\providecommand{\href}[2]{#2}

\end{document}